\documentclass[10pt,final]{article}
\usepackage{a4}
\usepackage{amsmath}%
\usepackage{amstext}%
\usepackage{amssymb}%
\usepackage{comment}
\usepackage{listings}

\usepackage[final]{graphicx}
\usepackage{subcaption}

\usepackage[margin=0.5in]{geometry}

\newtheorem{theorem}{Theorem}

\newtheorem{axiom}{Axiom}

\newtheorem{definition}[axiom]{Definition}

\newtheorem{lemma}[theorem]{Lemma}

\newtheorem{proposition}[theorem]{Proposition}
\newtheorem{rem}[theorem]{Remark}
\newenvironment{remark}{\rem\rm}{\endrem}

\newcounter{unnumber}

\newtheorem{prf}[unnumber]{Proof.}
\newenvironment{proof}{\prf\rm}{\hfill{$\blacksquare$}\endprf}
\newcommand{\R}{\mathbb{R}}%
\newcommand{\N}{\mathbb{N}}%
\newcommand{\ol}{\overline}%

\newcommand{\n}{{\nabla}}

\newcommand{\ds}{\displaystyle}
\newcommand{\To}{\longrightarrow}
\def\a{\alpha}

\def\b{\beta}

\def\g{\gamma}

\def\l{\lambda}
\def\<{\langle}
\def\>{\rangle}
\DeclareMathOperator*\prox{prox}%
\DeclareMathOperator*\argmin{argmin}

\DeclareMathOperator*\crit{crit}

\author{Cristian Daniel Alecsa\thanks{Tiberiu Popoviciu Institute of Numerical Analysis, Romanian Academy, Cluj-Napoca, Romania and  Department of Mathematics, Babes-Bolyai University, Cluj-Napoca, Romania, e-mail:  cristian.alecsa@math.ubbcluj.ro} \and Szil\'{a}rd Csaba L\'{a}szl\'{o} \thanks{Technical University of Cluj-Napoca, Department of Mathematics, Str. Memorandumului nr. 28, 400114 Cluj-Napoca, Romania, e-mail: laszlosziszi@yahoo.com} \and Titus Pin\c ta\thanks{Department of Mathematics, Babes-Bolyai University, Cluj-Napoca, Romania, e-mail:  titus.pinta@gmail.com} }

\title{An extension of the second order dynamical system that models Nesterov's convex gradient method\thanks{This work was supported by a grant of Ministry of Research and Innovation, CNCS - UEFISCDI, project number PN-III-P1-1.1-TE-2016-0266.}}

\begin{document}
\maketitle

\noindent \textbf{Abstract.}  In this paper we deal with a general second order continuous dynamical system associated to a convex minimization problem with a Fr\`echet differentiable objective function. We show that inertial algorithms, such as Nesterov's algorithm, can be obtained via the natural explicit discretization from our dynamical system.  Our dynamical system can be viewed as a perturbed version of the heavy ball method with vanishing damping, however the perturbation is made in the argument of the gradient of the objective function. This perturbation seems to have a smoothing effect for the energy error and eliminates the oscillations obtained for this error in the case of the heavy ball method with vanishing damping, as some numerical experiments show. We prove that the  value of the objective function in a generated trajectory converges in order $\mathcal{O}(1/t^2)$ to the global minimum of the objective function. Moreover, we obtain that a  trajectory  generated by the dynamical system converges to a minimum point of the objective function.  
\vspace{1ex}

 \vspace{1ex}


\section{Introduction}


Since Su, Boyd and Cand\`es  in  \cite{su-boyd-candes} showed that Nesterov's accelerated convex gradient method  has the exact limit the second order differential equation that governs the heavy ball system with vanishing damping, that is,
 \begin{equation}\label{ee11}
\ddot{x}(t)+\frac{\a}{t}\dot{x}(t)+\n g(x(t))=0,\,\,x(t_0)=u_0,\,\dot{x}(t_0)=v_0,\,t_0>0,\,u_0,v_0\in\R^m,
\end{equation}
with $\a=3$, the latter system has been intensively studied in the literature in connection to the minimization problem
$\inf_{x\in\R^m}g(x).$ Here $g:\R^m\To\R$ is a convex Fr\`echet differentiable function with Lipschitz continuous gradient.

In \cite{su-boyd-candes} the authors proved that $$g(x(t))-\min g=\mathcal{O}\left(\frac{1}{t^2}\right)$$
for every $\a\ge 3,$ however they did not show the convergence of a generated trajectory to a minimum of the objective function $g.$

In  \cite{att-c-p-r-math-pr2018}, Attouch, Chbani,  Peypouquet and Redont considered the case $\a>3$ in \eqref{ee11}, and showed that  the  generated  trajectory $x(t)$  converges  to  a  minimizer  of $g$ as $t\To+\infty$. Actually in \cite{att-c-p-r-math-pr2018} the authors considered the perturbed version of the heavy ball system with vanishing damping, that is,
\begin{equation}\label{ee12}
\ddot{x}(t)+\frac{\a}{t}\dot{x}(t)+\n g(x(t))=h(t),\,\,x(t_0)=u_0,\,\dot{x}(t_0)=v_0,\,t_0>0,\,u_0,v_0\in\R^m,
\end{equation}
where $h:[t_0,+\infty)\To\R$ is a small perturbation therm that satisfies $\int_{t_0}^{+\infty}t\|h(t)\|dt<+\infty.$ Beside the convergence of a generated trajectory $x(t)$ to a minimizer of $g$,  they showed that also in this case the convergence rate of the objective function along the trajectory, that is $g(x(t))-\min g,$ is of order $\mathcal{O}\left(\frac{1}{t^2}\right)$.

Another perturbed version of \eqref{ee11} was studied by Attouch,   Peypouquet and Redont in \cite{att-p-r-jde2016}. They assumed that the objective $g$ is twice continuously differentiable and the perturbation of their system is made at the damping therm. More precisely, they studied the dynamical system with Hessian driven damping
\begin{equation}\label{ee13}
\ddot{x}(t)+\frac{\a}{t}\dot{x}(t)+\b\n^2g(x(t))\dot{x}(t)+\n g(x(t))=0,\,\,x(t_0)=u_0,\,\dot{x}(t_0)=v_0,\,t_0>0,\,u_0,v_0\in\R^m,
\end{equation}
where $\a>0$ and $\b>0.$
In case $\a>3,\,\b>0$ they showed the convergence of a generated trajectory to a minimizer of $g$. Moreover, they obtained that in this case the convergence rate of the objective function along the trajectory, that is $g(x(t))-\min g,$ is of order ${o}\left(\frac{1}{t^2}\right)$.

Further,  Attouch, Chbani and Riahi in \cite{att-c-r-arx2017} studied the  subcritical  case $\a\le 3$ and  they proved that in this case the  convergence rates of the objective function  $g$ along  the  trajectory generated by \eqref{ee11}, i.e $g(x(t))-\min g,$ is of order $\mathcal{O}\left(\frac{1}{t^{\frac23\a}}\right)$.

Another approach is due to Aujol, Dossal and Rondepierre \cite{ADR}, who assumed that beside convexity, the objective $g$ in \eqref{ee11} satisfies some geometrical conditions, such as the {\L}ojasiewicz property. The importance of their results obtained in \cite{ADR} is underlined by the fact that applying the classical Nesterov scheme on a convex objective function without studying its geometrical properties may lead to sub-optimal algorithms.

 It is worth mentioning the work of Aujol and Dossal  \cite{AD},  who did not assumed the convexity of $g,$ but the convexity of the function $(g(x(t))-g(x^*))^\b$, where $\b$ is strongly related to the damping parameter $\a$ and $x^*$ is a global minimizer of $g.$ Under these assumptions, they obtained some general convergence rates and also the convergence of the generated trajectories of \eqref{ee11}. In case $\b=1$ they results reduce to the results obtained in \cite{att-c-p-r-math-pr2018,att-p-r-jde2016,att-c-r-arx2017}.

However, the convergence of the  trajectories generated by  the continuous heavy ball system with vanishing damping \eqref{ee11}, in the general case when $g$ is nonconvex is still an open question.
 Some important steps in this direction have been made in \cite{BCL-AA} (see also \cite{BCL}), where convergence of the trajectories of a perturbed system,  have been obtained in a nonconvex setting. More precisely in \cite{BCL-AA}  is considered the system
  \begin{equation}\label{eee11}
\ddot{x}(t)+\left(\g+\frac{\a}{t}\right)\dot{x}(t)+\n g(x(t))=0,\, x(t_0)=u_0,\,\dot{x}(t_0)=v_0,
\end{equation}
where $t_0>0,u_0,v_0\in\R^m,\,\g>0,\,\a\in\R.$  Note that here $\a$ can take nonpositive values. For $\a=0$ we recover the dynamical system studied in \cite{BBJ}. According to \cite{BCL-AA}, the trajectory  generated by the dynamical system \eqref{eee11} converges  to a critical point of $g,$ if a regularization of $g$ satisfies the Kurdyka-{\L}ojasiewicz property.

Further results concerning the heavy ball method and its extensions can be found in \cite{alv-att-bolte-red,AC,AGR,BM,CAG1,CAG2}.

\subsection{An extension of the heavy ball method and the Nesterov type algorithms obtained via explicit discretization}

What one can notice concerning the heavy ball system and its variants is, that despite of the result of Su et al. \cite{su-boyd-candes}, this system will never give through the natural implicit/explicit discretization the Nesterov algorithm. This is due to the fact that the gradient of $g$ is evaluated in $x(t),$ and this via discretization will become $g(x_n)$, (or $g(x_{n+1})$) and never of the form $g(y_n),\,y_n=x_n+\a_n(x_n-x_{n-1})$ as Nesterov's gradient method requires.
Another observation is that using the same approach as in \cite{su-boyd-candes}, one can show, see \cite{L}, that  \eqref{ee11}, (and also \eqref{eee11}),  models beside Nesterov's algorithm other algorithms too. In this paper we overcome the deficiency emphasized above, by introducing a dynamical system that via explicit discretization leads to inertial algorithms of gradient type. To this end, let us consider the optimization problem
\begin{equation}\label{opt-pb} (P) \ \inf_{x\in\R^m}g(x) \end{equation}
where $g:\R^m\To \R$ is a convex Fr\'{e}chet differentiable, function with $L_g$-Lipschitz continuous gradient, i.e. there exists $L_g\ge 0$ such that $\|\n g(x)-\n g(y)\|\le L_g\|x-y\|$ for all $x,y\in \R^n.$

We associate to \eqref{opt-pb} the following second order dynamical system:
\begin{equation}\label{dysy}
\left\{
\begin{array}{ll}
\ddot{x}(t)+\frac{\a}{t}\dot{x}(t)+\nabla g\left(x(t)+\left(\g+\frac{\b}{t}\right)\dot{x}(t)\right)=0\\
x(t_0)=u_0,\,\dot{x}(t_0)=v_0,
\end{array}
\right.
\end{equation}
where $u_0,v_0\in \R^m,$ $t_0\ge0$ and $\a>0,\,\b\in\R,\,\g\ge 0.$
\begin{remark}\label{rex}
The connection of \eqref{dysy} with the  heavy ball system with vanishing damping  \eqref{ee11} is obvious, the latter one can be obtained from \eqref{dysy} for $\g=\b=0.$
 The study of the dynamical system \eqref{dysy} in connection to the optimization problem \eqref{opt-pb} is motivated by the following facts:
 
  {\bf 1.} The dynamical system \eqref{dysy}  leads via explicit discretization to inertial algorithms.  In particular Nesterov's algorithm can be obtained via this natural discretization. 
  
  {\bf 2.} A generated trajectory and the objective function value in this trajectory in general have a better convergence behaviour than a trajectory generated by the heavy ball system with vanishing damping \eqref{ee11}, as some numerical examples shows.
  
  {\bf 3.} The same numerical experiments reveal that the perturbation term $\left(\g+\frac{\b}{t}\right)\dot{x}(t)$ in the argument of the  gradient of the objective function $g$ has a smoothing effect and annihilates the oscillations obtained in case of the dynamical system \eqref{ee11} for the errors 
  $g(x(t)-\min g$ and $\|x(t)-x^*\|,$ where $x^*$ is a minimizer of $g.$
  
  {\bf 4.} A trajectory  $x(t)$ generated by the dynamical system \eqref{dysy} ensures the convergence rate of order $\mathcal{O}\left(\frac{1}{t^2}\right)$ for the decay $g\left(x(t)+\left(\g+\frac{\b}{t}\right)\dot{x}(t)\right)-\min g$, provided it holds that $\a>3,\,\g> 0,\,\b\in\R$  or $\a\ge3,\,g=0,\,\b\ge 0.$
 
 {\bf 5.} A trajectory  $x(t)$ generated by the dynamical system \eqref{dysy} ensures the same convergence rate of order $\mathcal{O}\left(\frac{1}{t^2}\right)$ for the decay $g(x(t))-\min g$  as the heavy ball method with vanishing damping, for the cases $\a>3,\,\g> 0,\,\b\in\R$ and $\a>3,\,\g=0,\,\b\ge 0.$ 
    
 {\bf 6.} The convergence of a generated trajectory $x(t)$ to a minimizer of the objective function $g$ can be obtained in case $\a>3,\,\g>0$ and $\b\in\R$ and also in the case $\a>3,\,\g=0$ and $\b\ge 0.$ 
  \end{remark}
  
  \begin{remark}
  
  Nevertheless, in case $\g=0$ and $\b<0$ the dynamical system \eqref{dysy} can generate periodical solutions, hence the convergence of a generated trajectory to a minimizer of the objective is hopeless. To illustrate this fact, for $\b<0,\,\a>0,\,\g=0$ consider the strongly convex objective function $g:\R\To\R,$ $g(x)=\frac{\a}{-2\b} x^2$. Then, taking into account that $\g=0$ the dynamical system \eqref{dysy} becomes
 \begin{equation}\label{exdysy}
\left\{
\begin{array}{ll}
\ddot{x}(t)+\frac{\a}{-\b}x(t)=0,\\
x(0)=0,\,\dot{x}(0)=\sqrt{\frac{\a}{-\b}}.
\end{array}
\right.
\end{equation}
Now, the periodical function $x(t)=\sin\sqrt{\frac{\a}{-\b}} t$ is a solution of \eqref{exdysy}, consequently do not exist the limit $\lim_{t\To+\infty}x(t).$
\end{remark}

We emphasize, that despite the novelty of the dynamical system \eqref{dysy}, its formulation is natural since by explicit discretization leads to inertial gradient methods, in particular the famous Polyak and Nesterov numerical schemes can be obtained from \eqref{dysy}. For other inertial algorithms of gradient type we refer to \cite{BCL1,ALV,L}.

Indeed, explicit discretization of \eqref{dysy}, with the constant stepsize $h$,  $t_n=nh,\,\,x_n=x(t_n)$ leads to
$$\frac{x_{n+1}-2x_n+x_{n-1}}{h^2}+\frac{\a}{nh^2}(x_n-x_{n-1})+\n g\left(x_n+\left(\frac{\g}{h}+\frac{\b}{nh^2}\right)(x_n-x_{n-1})\right)=0.$$

Equivalently, the latter equation can be written as
\begin{equation}\label{discretization}
x_{n+1}=x_n+\left(1-\frac{\a}{n}\right)(x_n-x_{n-1})-h^2\n g\left(x_n+\left(\frac{\g}{h}+\frac{\b}{nh^2}\right)(x_n-x_{n-1})\right).
\end{equation}

Now, setting $h^2=s$ and denoting the constants $\frac{\g}{h}$  and $\frac{\b}{h^2}$ still with $\g$ and $\b$, we get the following general inertial algorithm:

Let $x_0,x_{-1}\in\R^m$ and for all $n\in\N$ consider the sequences
\begin{equation}\label{alggentemp}
\left\{
\begin{array}{lll}
y_n=x_n+\left(1-\frac{\a}{n}\right)(x_n-x_{n-1})\\
z_n=x_n+\left({\g}+\frac{\b}{n}\right)(x_n-x_{n-1})\\
x_{n+1}=y_n-s\n g\left(z_n\right).
\end{array}
\right.
\end{equation}
However, from a practical point of view it is more convenient to work with the following equivalent formulation:
Let $x_0,x_{-1}\in\R^m$ and for all $n\in\N$ consider the sequences
\begin{equation}\label{alggen}
\left\{
\begin{array}{lll}
y_n=x_n+\frac{n}{n+\a}(x_n-x_{n-1})\\
\\
z_n=x_n+\frac{\g n+\b}{n+\a}(x_n-x_{n-1})\\
\\
x_{n+1}=y_n-s\n g\left(z_n\right),
\end{array}
\right.
\end{equation}
where $\a>0,\,\b\in\R$ and $\g\ge 0.$

\begin{remark}
Notice that for $\g=\b=0$ we obtain a variant of Polyak's algorithm \cite{poly} and for $\g=1$ and $\b=0$ we obtain Nesterov's algorithm \cite{nesterov83}. An interesting fact is that Algorithm \eqref{alggen} allows different inertial steps and this approach seems to be new in the literature.
\end{remark}

\begin{remark}
Independently to us, very recently, a  system similar to \eqref{dysy} was studied by Muehlebach and Jordan in \cite{MJ} and they show that Nesterov's accelerated gradient method can be obtained from their system via a semi-implicit Euler discretization scheme. They considered a constant damping instead of $\frac{\a}{t}$ and also took $\b=0.$ Further they also treated the ODE
$$\ddot{x}(t)+\frac{3}{t+2}\dot{x}(t)+s\nabla g\left(x(t)+\frac{t-1}{t+2}\dot{x}(t)\right)=0$$ which for $s=1$ is obviously equivalent to the particular case of the governing ODE from \eqref{dysy}, obtained for $\a=3,\,\g=1,\,\b=-3.$ 
However, the freedom of controlling the parameters $\b$ and $\g$ in \eqref{dysy} is essential as the next numerical experiments show.
\end{remark}

\subsection{Some numerical experiments}

In this section we consider two numerical experiments for the trajectories generated by the dynamical system \eqref{dysy} for a strongly convex and for a convex but not strongly convex objective function.

Everywhere in the following numerical experiments we consider the continuous time dynamical system \eqref{dysy}, solved numerically with a Runge Kutta 4-5 (ode45) adaptive method in MATLAB. We solved the dynamical system with ode45 on the interval $[1, 100]$  and the plot in Fig. 1 - Fig. 2 show the energy error $|g(x(t)) - g(x^*)|$ on the left, and the iterate error $\|x(t) - x^*\|$ on the right.

We show the evolution of the two errors with respect to different values for $\alpha$, $\beta$ and $\gamma$, including the case that yields the Heavy Ball with Friction. One can observe that the best choice is not $\g=\b=0$ which is the case of heavy ball system with vanishing damping \eqref{ee11}.

{\bf 1.} Consider the function $g:\R^2\To\R,$
$$g(x,y)=2x^2 + 5y^2 -4x+10y+7.$$
Then, $g$ is a strongly convex function, $\n g(x,y)=(4x-4,10y+10)$, further $x^*=(1,-1)$ is the unique minimizer of $g$ and $g^*=g(1,-1)=0.$

We compare the convergence behaviour of the generated trajectories of \eqref{dysy} by taking into account the following instances.

 \begin{center}
    \begin{tabular}{|l|c|r|l|}
      $\alpha$ & $\beta$ & $\gamma$ & \text{Color}\\ 
      \hline
      3   & 0   & 0 & \text{blue}\\ 
      3.1 & 1   & 0 & \text{red}\\ 
      3.1 & 0 & 0.5 & \text{yellow}\\ 
      3.1 & 1 & 1 & \text{purple}\\ %
      3.1 & -1 & 1 & \text{green}\\ %
    \end{tabular}
  \end{center}

The result are depicted in Figure.1 for the starting points $u_0=v_0=(-5,30)$ and $u_0=v_0=(5,-30)$, respectively.
\begin{figure}
\begin{subfigure}{.5\textwidth}
  \centering
  \includegraphics[width=.99\linewidth]{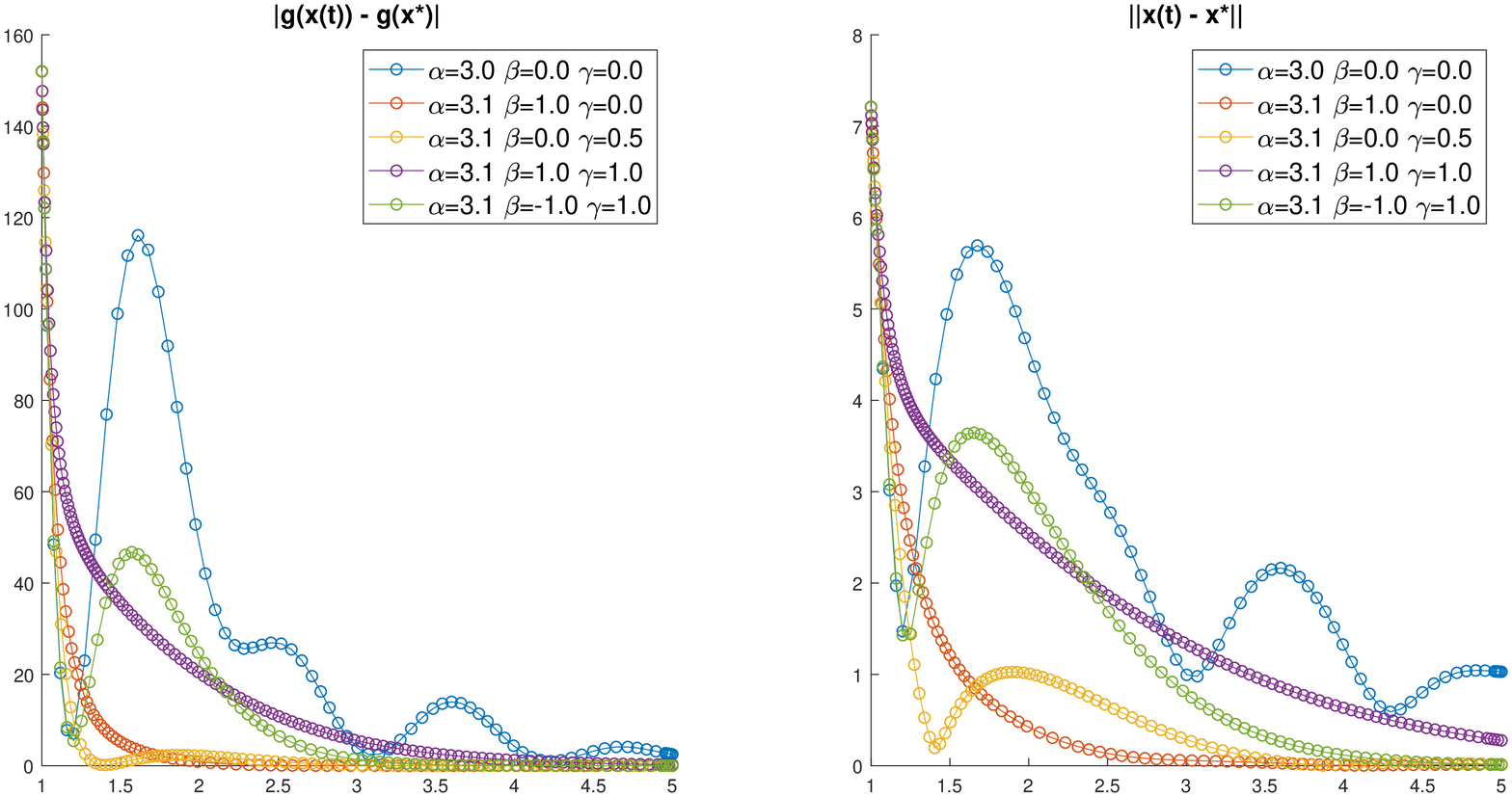}
  \caption{$u_0=v_0=(-5,30).$}
  \label{fig1:sfig11}
\end{subfigure}
\begin{subfigure}{.5\textwidth}
  \centering
  \includegraphics[width=.99\linewidth]{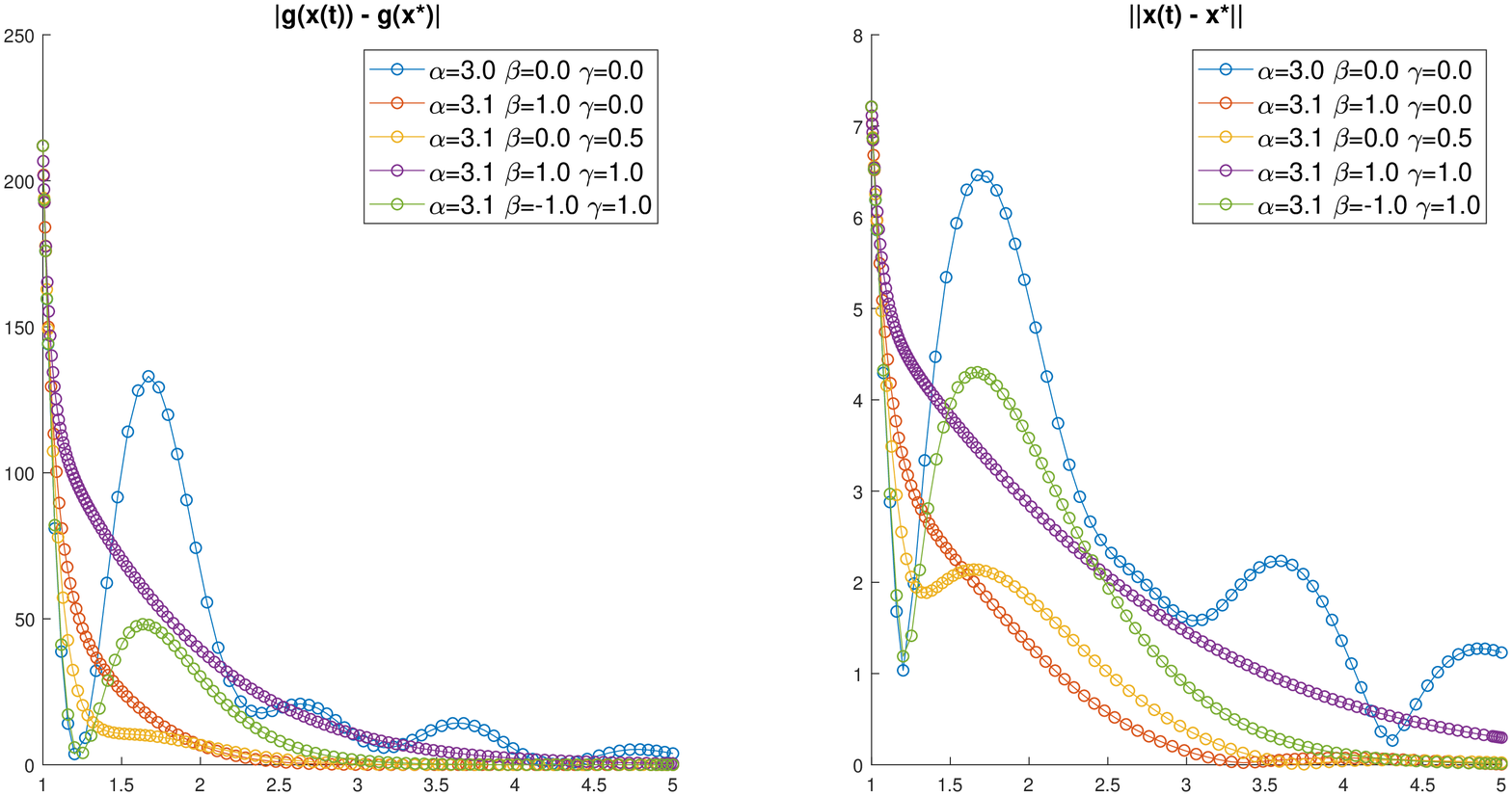}
  \caption{ $u_0=v_0=(5,-30).$}
  \label{fig1:sfig12}
\end{subfigure}
 \caption{Error analysis with different parameters in dynamical system (4) for a strongly convex objective function.}
\end{figure}

{\bf 2.} In the next experiment we consider the convex, but not strongly convex function $g:\R^2\To\R,$
$$g(x,y)=x^4 + 5y^2 -4x-10y+8.$$
Then,  $\n g(x,y)=(4x^3-4,10y-10)$, further $x^*=(1,1)$ is the unique minimizer of $g$ and $g^*=g(1,1)=0.$

We compare the convergence behaviour of the generated trajectories of \eqref{dysy} by taking into account the following instances.

 \begin{center}
    \begin{tabular}{|l|c|r|l|}
      $\alpha$ & $\beta$ & $\gamma$ & \text{Color}\\ 
      \hline
      3.1   & 0   & 0 & \text{blue}\\ 
      3.1 & 2   & 0 & \text{red}\\ 
      3.1 & 0 & 1 & \text{yellow}\\ 
      3.1 & 0.5 & 0.5 & \text{purple}\\ %
      3.1 & -0.5 & 1 & \text{green}\\ %
    \end{tabular}
  \end{center}

The result are depicted in Figure.2 for the starting points $u_0=v_0=(-1,5)$ and $u_0=v_0=(2,-2)$, respectively.

\begin{figure}
\begin{subfigure}{.5\textwidth}
  \centering
  \includegraphics[width=.99\linewidth]{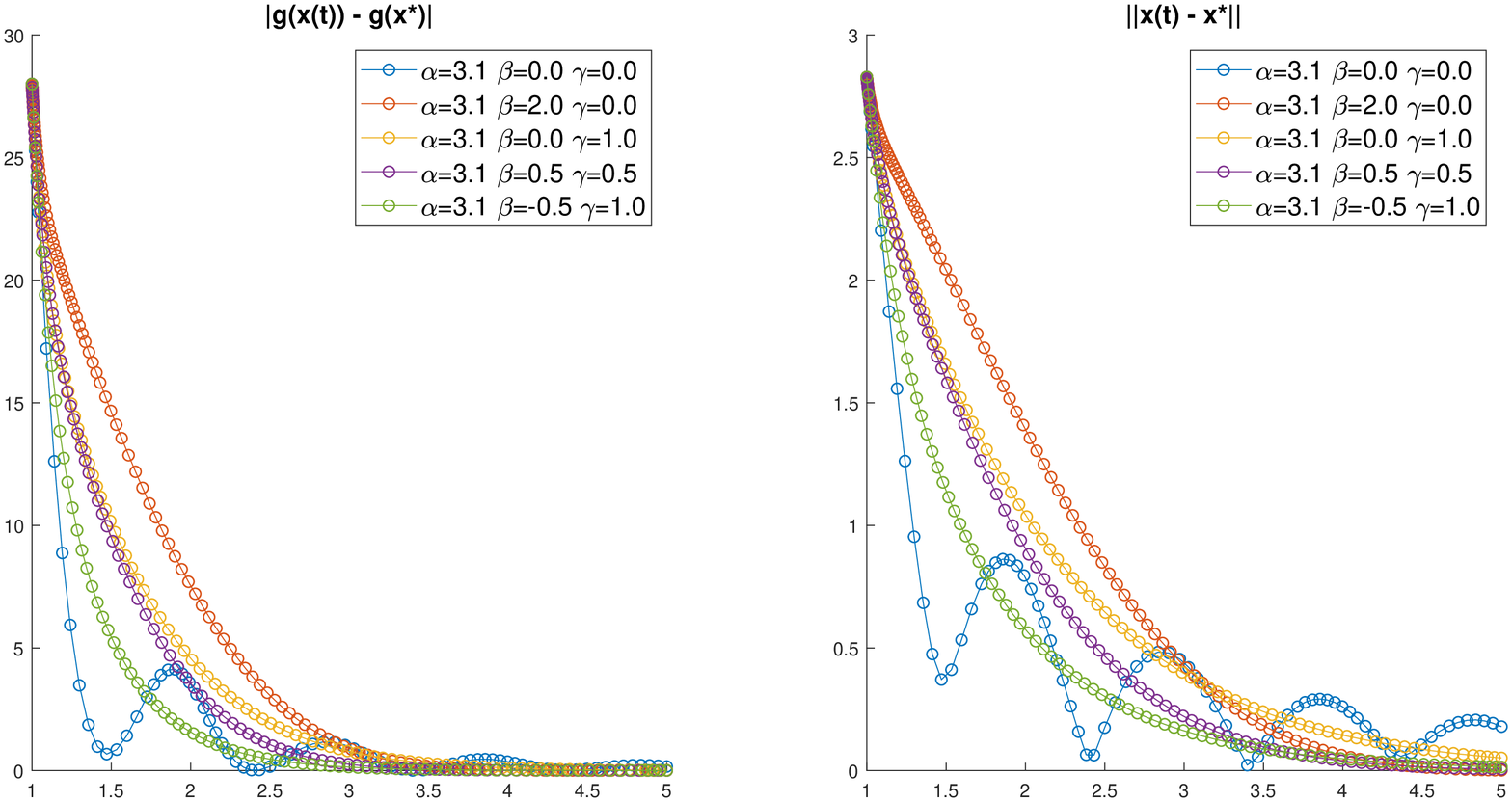}
  \caption{$u_0=v_0=(-1,5).$}
  \label{fig2:sfig21}
\end{subfigure}
\begin{subfigure}{.5\textwidth}
  \centering
  \includegraphics[width=.99\linewidth]{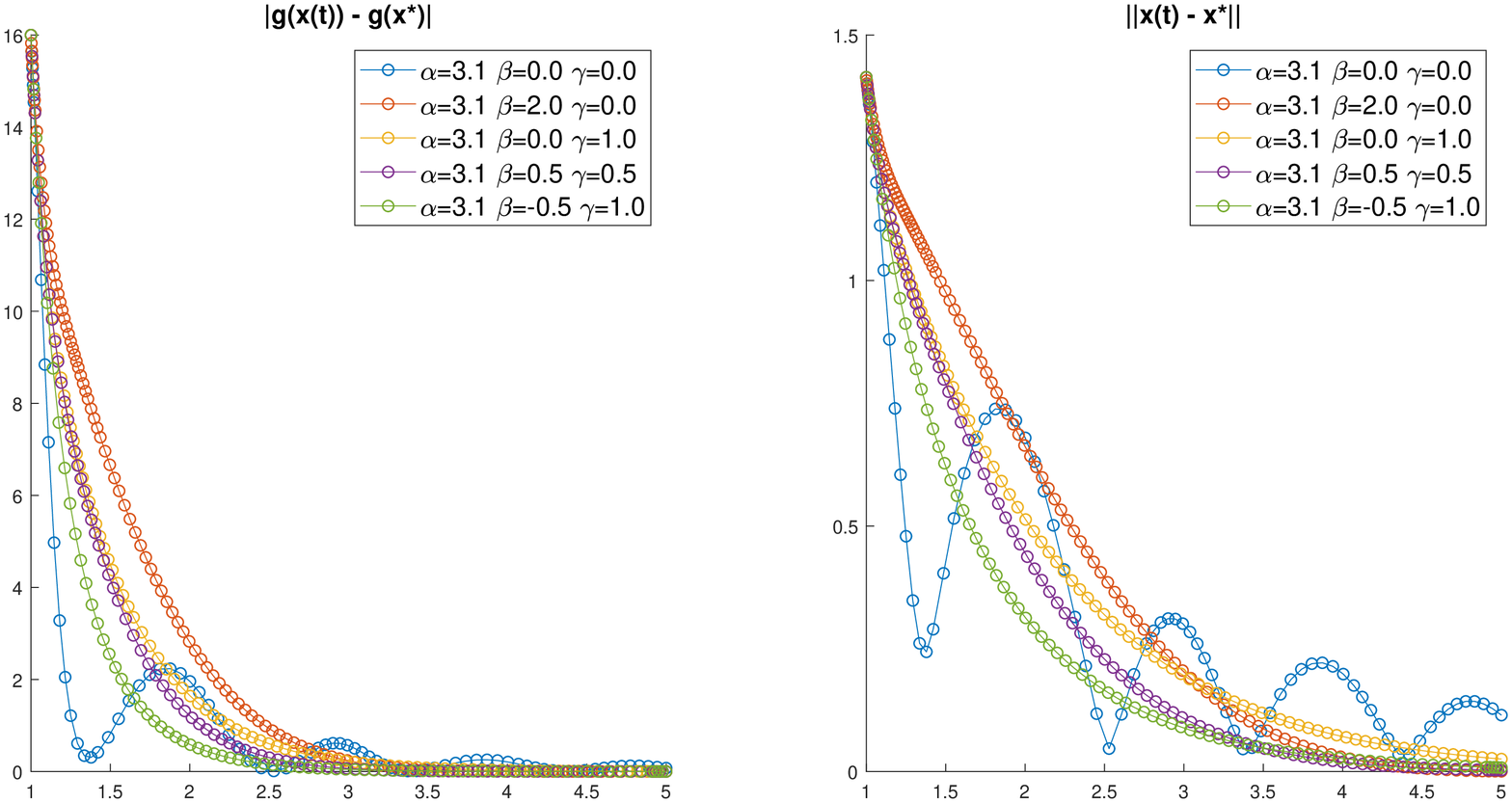}
  \caption{ $u_0=v_0=(2,-2).$}
  \label{fig2:sfig22}
\end{subfigure}
 \caption{Error analysis with different parameters in dynamical system (4) for a convex, but not strongly convex, objective function.}
\end{figure}

\begin{remark} Observe that in all cases the trajectories generated by dynamical system \eqref{dysy} have a better behaviour than the trajectories generated by the heavy ball system with vanishing damping \eqref{ee11}. As we have emphasized before, it seems that the perturbation $\left(\g+\frac{\b}{t}\right)\dot{x}(t)$ in the argument of the gradient of the objective function has a smoothing effect. The choice of the parameters $\g$ and $\b$ will be validated by the theoretical results from Section 3, where we show that in case $\a>3,\,\g>0,\,\b\in\R$ and also in the case $\a>3,\,\g=0,\,\b\ge0$ the energy error $g(x(t)-\min g$  is of order $\mathcal{O}\left(\frac{1}{t^2}\right)$ just as the case of heavy ball method. Further, for the values $\a>3,\,\g>0,\,\b\in\R$ and for $\a>0,\,\g=0,\,\b\ge0$  we are able to show that a generated trajectory $x(t)$ converges to a minimum of the objective function $g.$
\end{remark}

\subsection{The organization of the paper}

The outline of the paper is the following. In the next section we show the existence and uniqueness of the trajectories generated by the system \eqref{dysy}. Further we show that the third order derivative exists almost for every $t\ge t_0$ and we  give some estimates of the third order derivative in terms of velocity and acceleration. We do not assume the convexity of the objective function $g$ in these results. However, it seems that the assumption that $g$ has Lipschitz continuous gradient is essential in obtaining existence and uniqueness of the generated trajectories of the dynamical system \eqref{dysy}. In Section 3 we deal with the convergence analysis of the generated trajectories. We introduce a general energy functional, which will play the role of a Lyapunov function associated to the dynamical system \eqref{dysy}. We show convergence of the generated trajectories and also a rate of order $\mathcal{O}(1/t^2)$ for the decay $g\left(x(t)+\left(\g+\frac{\b}{t}\right)\dot{x}(t)\right)-\min g.$ Further, we show that also the error $g(x(t))-\min g$ has a rate of order $\mathcal{O}(1/t^2)$. Finally, we show the convergence of the generated trajectories to a minimum of the objective function $g.$ In section 4 we conclude our paper and we present some possible related future researches.

\section{Existence and uniqueness}
\subsection{On strong global solutions}
The first step toward our existence and uniqueness result obtained in the present section concerns the definition of a strong global solution of the dynamical system \eqref{dysy}.

\begin{definition}
 We call the function $x:[t_0, + \infty) \to \R^m$ a strong global solution of the dynamical system (\ref{dysy}) if  satisfies the following properties:
\begin{align*}
& (i) \quad x, \dot{x} : [t_0, + \infty) \to \R^m \text{ are locally absolutely continuous}; \\
& (ii) \quad \ddot{x}(t) + \frac{\a}{t} \dot{x}(t) + \nabla g \left( x(t)+ \left(\g+\frac{\b}{t} \right)  \dot{x}(t) \right)=0 \text{ for almost every } t \geq t_0; \\
& (iii) \quad x(t_0) = u_0 \text{ and } \dot{x}(t_0) = v_0.
\end{align*}
\end{definition}
For brevity reasons, we recall that a mapping $x : [t_0, + \infty) \to \R^m$ is called locally absolutely continuous if it is absolutely continuous on every compact interval $[t_0, T]$, where $T > t_0$. Further, we have the following equivalent characterizations for an absolutely continuous function $x:[t_0,T]\To\R^m$, (see, for instance, \cite{att-sv2011, abbas-att-sv}):
\begin{itemize}
\item[(a)] there exists $y : [t_0, T] \to \R^m$ an integrable function, such that
$$ x(t) = x(t_0) + \int\limits_{t_0}^{t} y(t) ds, \forall t \in [t_0, T];$$
\item[(b)] $x$ is a continuous function and its distributional derivative is Lebesque integrable on the interval $[t_0, T];$
\item[(c)] for every $\varepsilon > 0$, there exists $\eta > 0$, such that for every finite family $I_k = (a_k, b_k)$ from $[t_0, T]$, the following implication is valid :
\begin{align*}
\left[ I_k \cap I_j = \emptyset \text{ and } \sum\limits_{k} |b_k-a_k| < \eta \right] \Rightarrow \left[ \sum\limits_k \| x(b_k) - x(a_k) \| < \varepsilon \right].
\end{align*}
\end{itemize}


\begin{remark}\label{RemarkDiff}
Let $x : [t_0, +\infty) \to \R^m$ be a locally absolutely continuous function. Then $x$ is differentiable almost everywhere and its derivative coincides with its distributional derivative almost everywhere. On the other hand, we have the equality $\dot{x}(t) = y(t)$ for almost every $t \in [t_0, +\infty)$, where $y = y(t)$ is defined at the integration formula (a).
\end{remark}

The first result of the present section concerns the existence and uniqueness of the trajectories generated by the dynamical system (\ref{dysy}). We prove existence and uniqueness of a strong global solution of \eqref{dysy} by making use of the Cauchy-Lipschitz-Picard Theorem for absolutely continues trajectories (see for example \cite[Proposition 6.2.1]{haraux}, \cite[Theorem 54]{sontag}). The key argument is that one can rewrite \eqref{dysy} as a particular first order dynamical system in a suitably chosen product space (see also \cite{alv-att-bolte-red,att-c-r,BCL-AA,BCL-MM}).

\begin{theorem}\label{exuniq}
Let $(u_0, v_0) \in \R^m \times \R^m$. Then, the dynamical system (\ref{dysy}) admits a unique strong global solution.
\end{theorem}

\begin{proof}
By making use of the notation $X(t)=(x(t),\dot{x}(t))$ the system \eqref{dysy} can be rewritten as a first order dynamical system:
\begin{equation}\label{dysy4}
\left\{
\begin{array}{ll}
\dot{X}(t)=F(t,X(t))\\
X(t_0)=(u_0,v_0),
\end{array}
\right.
\end{equation}
where $F:[t_0,+\infty)\times \R^m\times\R^m\To \R^m\times\R^m,\, F(t,u,v)=\left(v, -\frac{\a}{t}v-\nabla g\left(u+\left(\g+\frac{\b}{t}\right)v\right)\right).$

 The existence and uniqueness of a strong global solution of \eqref{dysy} follows according to the Cauchy-Lipschitz-Picard Theorem applied to the first order dynamical system \eqref{dysy4}. In order to prove the existence and uniqueness of the trajectories generated by \eqref{dysy4} we show the following:

 (I) For every $t\in[t_0,+\infty)$ the mapping $F(t, \cdot, \cdot)$ is $L(t)$-Lipschitz continuous and $L(\cdot)\in L^1_{loc}([t_0,+\infty))$.

 (II) For all $u,v\in {\R^m}$ one has $F(\cdot,u,v)\in L^1_{loc}([t_0,+\infty),{\R^m}\times {\R^m})$ .
\\

Let us prove {(I)}. Let $t \in [t_0, +\infty)$ be fixed and consider the pairs $(u,v)$ and $(\bar{u}, \bar{v})$ from $\R^m \times \R^m$. Using the Lipschitz continuity of $\n g$ and the obvious inequality $\|A+B\|^2\le 2\|A\|^2+2\|B\|^2$ for all $A,B\in\R^m$, we make the following estimations :
\begin{align*}
\|F(t,u,v)-F(t,\ol u,\ol v)\|&=\sqrt{\|v-\ol v\|^2+\left\|\frac{\a}{t}(\ol v-v)+\nabla g\left(\ol u+\left(\g+\frac{\b}{t}\right)\ol v\right)-\nabla g\left(u+\left(\g+\frac{\b}{t}\right)v\right)\right\|^2} \\
& \le\sqrt{\left(1+2\left(\frac{\a}{t}\right)^2\right)\|v-\ol v\|^2+2L_g^2\left\|(u-\ol u)+\left(\g+\frac{\b}{t}\right)(v-\ol v)\right\|^2} \\
 &\le\sqrt{\left(1+2\left(\frac{\a}{t}\right)^2+4L_g^2\left(\g+\frac{\b}{t}\right)^2\right)\|v-\ol v\|^2+4L_g^2\left\|u-\ol u\right\|^2} \\
 &\le \sqrt{1+4L_g^2+2\left(\frac{\a}{t}\right)^2+4L_g^2\left(\g+\frac{\b}{t}\right)^2}\sqrt{\|v-\ol v\|^2+\|u-\ol u\|^2}  \\
&=\sqrt{1+4L_g^2+2\left(\frac{\a}{t}\right)^2+4L_g^2\left(\g+\frac{\b}{t}\right)^2}\,\|(u,v)-(\ol u,\ol v)\|.
\end{align*}

By employing the notation $L(t)= \sqrt{1+4L_g^2+2\left(\frac{\a}{t}\right)^2+4L_g^2\left(\g+\frac{\b}{t}\right)^2}$, we have that
$$\| F(t,u,v) - F(t, \bar{u}, \bar{v}) \| \leq L(t) \cdot \| (u,v) - (\bar{u}, \bar{v}) \|.$$
Obviously the function $t \To L(t)$ is  continuous  on $[t_0, +\infty)$, hence $L(\cdot)$ is integrable on $[t_0,T]$ for all $t_0<T<+\infty$.

For proving {(II)} consider $(u,v) \in \R^m \times \R^m$  a fixed pair of elements and let $T > t_0$. We consider the following estimations:
\begin{align*}
 \int_{t_0}^T\|F(t,u,v)\|dt &=\int_{t_0}^T \sqrt{\|v\|^2+\left\|\frac{\a}{t}v+\nabla g\left(u+\left(\g+\frac{\b}{t}\right)v\right)\right\|^2}dt  \\
&\le\int_{t_0}^T \sqrt{\left(1+2\left(\frac{\a}{t}\right)^2\right)\|v\|^2+4\left\|\nabla g\left(u+\left(\g+\frac{\b}{t}\right)v\right)-\n g(u)\right\|^2+4\|\n g(u)\|^2}dt  \\
&\le\int_{t_0}^T \sqrt{\left(1+2\left(\frac{\a}{t}\right)^2+4L_g^2\left(\g+\frac{\b}{t}\right)^2\right)\|v\|^2+4\|\n g(u)\|^2}dt  \\
&\le \sqrt{\|v\|^2+\|\nabla g(u)\|^2}\int_{t_0}^T \sqrt{5+2\left(\frac{\a}{t}\right)^2+4L_g^2\left(\g+\frac{\b}{t}\right)^2}dt
\end{align*}
and the conclusion follows by the continuity of the function $t \To \sqrt{5+2\left(\frac{\a}{t}\right)^2+4L_g^2\left(\g+\frac{\b}{t}\right)^2}$ on $[t_0,T]$.

The Cauchy-Lipschitz-Picard theorem guarantees existence and uniqueness of the trajectory of the first order dynamical system \eqref{dysy4} and thus of the second order dynamical system \eqref{dysy}.
\end{proof}

\begin{remark} Note that we did not use the convexity assumption imposed on $g$ in the proof of Theorem \ref{exuniq}. However, we emphasize that according to the proof of Theorem \ref{exuniq}, the assumption that $g$ has a Lipschitz continuous gradient is essential in order to obtain existence  and uniqueness of the trajectories generated by the dynamical system \eqref{dysy}.
\end{remark}

\subsection{On the third order derivatives}

In this section we show that the third order derivative of a strong global solution of the dynamical system \eqref{dysy} exists almost everywhere on $[t_0, + \infty)$. Further we give an upper bound estimate for the third order derivative in terms of velocity and acceleration of a strong global solution of \eqref{dysy}. For simplicity, in the proof of the following sequel we employ the $1$-norm on $\R^m \times \R^m$, defined as $\| (x_1, x_2) \|_{1} = \| x_1 \| + \| x_2 \|$, for all $x_1,x_2 \in \R^m$. Obviously one has
$$\frac{1}{\sqrt{2}}\| (x_1, x_2) \|_{1}\le \| (x_1, x_2) \|=\sqrt{\|x_1\|^2+\|x_2\|^2}\le\| (x_1, x_2) \|_{1},\mbox{ for all }x_1,x_2\in\R^m.$$

\begin{proposition}
For the starting points $(u_0,v_0) \in \R^m \times \R^m$ let $x$ be the unique strong global solution of the dynamical system (\ref{dysy}). Then, $\ddot{x}$ is locally absolutely continuous on $[t_0, + \infty)$, consequently the third order derivative $x^{(3)}$ exists almost everywhere on $[t_0, + \infty)$.
\end{proposition}

\begin{proof}
We  show that $\dot{X}(t) = (\dot{x}(t), \ddot{x}(t))$ is locally absolutely continuous, hence $\ddot{x}$ is also locally absolutely continuous. This implies by Remark \ref{RemarkDiff} that $x^{(3)}$ exists almost everywhere on $[t_0, + \infty)$.

Let $T > t_0$ and $s,t \in [t_0, T]$. We consider the following chain of inequalities :
\begin{align*}
\| \dot{X}(s) - \dot{X}(t) \|_{1} &= \| F(s,X(s)) - F(t,X(t)) \|_{1} \\
&= \left\| \left( \dot{x}(s) - \dot{x}(t) , \frac{\a}{t} \dot{x}(t)-\frac{\a}{s} \dot{x}(t) +\n g\left( x(t)+ \left(\g+\frac{\b}{t}\right) \dot{x}(t) \right) - \n g\left( x(s) + \left(\g+\frac{\b}{s}\right) \dot{x}(s) \right)\right) \right\|_{1}\\
&\le \| \dot{x}(s) - \dot{x}(t) \| + \left\| \frac{\a}{s} \dot{x}(s) - \frac{\a}{t} \dot{x}(t) \right\| \\
&+\left\| \n g\left( x(s) + \left(\g+\frac{\b}{s}\right) \dot{x}(s) \right) - \n g\left( x(t) + \left(\g+\frac{\b}{t}\right) \dot{x}(t) \right)  \right\|
\end{align*}
and by using the $L_g$-Lipschitz continuity of $\n g$ we obtain
\begin{align*}
\| \dot{X}(s) - \dot{X}(t) \|_{1}&\leq \| \dot{x}(s) - \dot{x}(t) \| + \left\| \frac{\a}{s} \dot{x}(s) - \frac{\a}{t} \dot{x}(t) \right\| + L_g \| x(s) - x(t)\| +L_g\left\|\left(\g+\frac{\b}{s}\right) \dot{x}(s) - \left(\g+\frac{\b}{t}\right) \dot{x}(t)\right\| \\
&\leq \| \dot{x}(s) - \dot{x}(t) \| + \dfrac{\alpha}{s} \| \dot{x}(s) - \dot{x}(t) \| + \left| \dfrac{\alpha}{s} - \frac{\a}{t} \right| \| \dot{x}(t) \|+ L_g \| x(s) - x(t) \| \\
&  + L_g  \left(\g+\frac{\b}{s}\right)\| \dot{x}(s) - \dot{x}(t) \| + L_g \left| \frac{\b}{s}-\frac{\b}{t}\right| \cdot \| \dot{x}(t) \| \\
&= \left( 1 + \dfrac{\alpha}{s} + L_g \left(\g+\frac{\b}{s}\right)\right) \| \dot{x}(s) -\dot{x}(t) \| + L_g \| x(s) - x(t) \| + \left( \a + L_g |\b| \right) \cdot \left| \dfrac{1}{s} - \dfrac{1}{t} \right|  \| \dot{x}(t) \|.
\end{align*}
Further, let us introduce the following additional notations:
$$L_1 := \max\limits_{s \in [t_0, T]}\left( 1 + \dfrac{\alpha}{s} + L_g \left(\g+\frac{\b}{s}\right)\right)= 1 + \dfrac{\alpha}{t_0} + L_g \left(\g+\frac{\b}{t_0}\right)\mbox{ and }L_2 :=  \left( \a + L_g |\b| \right) \max\limits_{t \in [t_0,T]} \| \dot{x}(t) \|.$$
Then, one has
\begin{align*}
\| \dot{X}(s) - \dot{X}(t) \|\le \| \dot{X}(s) - \dot{X}(t) \|_{1} \leq L_1 \| \dot{x}(s) - \dot{x}(t) \| + L_g \| x(s) - x(t) \| +  L_2  \left| \dfrac{1}{s} - \dfrac{1}{t} \right|.
\end{align*}
By the fact that $x$ is the strong global solution for the dynamical system (\ref{dysy}), it follows that $x$ and $\dot{x}$ are  absolutely continuous on the interval $[t_0,T]$. Moreover, the function $t \To \dfrac{1}{t}$ belongs to $C^1([t_0, T], \mathbb{R})$, hence it is also absolutely continuous on the interval $[t_0,T]$. Let $\varepsilon > 0$.  Then, there exists $\eta > 0$, such that for $I_k = (a_k,b_k) \subseteq [t_0,T]$ satisfying $I_k \cap I_j = \emptyset$ and $\sum\limits_{k} |b_k-a_k| < \eta$, we have that

$$ \sum\limits_{k} \| \dot{x}(b_k) - \dot{x}(a_k) \| < \dfrac{\varepsilon}{3L_1}, \, \sum\limits_{k} \| x(b_k) - x(a_k) \| < \dfrac{\varepsilon}{3L_g}\mbox{ and } \sum\limits_{k} \Big| \dfrac{1}{b_k} - \dfrac{1}{a_k} \Big| < \dfrac{\varepsilon}{3 L_2}.$$
Summing all up, we obtain
\begin{align*}
& \sum\limits_{k} \| \dot{X}(b_k) - \dot{X}(a_k) \| \leq L_1 \sum\limits_{k} \| \dot{x}(b_k) - \dot{x}(a_k) \| + L_g \sum\limits_{k} \| x(b_k) - x(a_k) \| + L_2  \sum\limits_{k} \left| \dfrac{1}{b_k} - \dfrac{1}{a_k} \right| < \varepsilon,
\end{align*}
consequently $\dot{X}$ is absolutely continuous on $[t_0,T]$. and the conclusion follows.
\end{proof}

Concerning an upper bound estimate of the third order derivative $x^{(3)}$ the following result holds.
\begin{lemma}
For the initial values $(u_0,v_0) \in \R^m \times \R^m$ consider $x$ the unique strong global solution of the second-order dynamical system (\ref{dysy}). Then, there  exists $K> 0$ such that for almost every $t \in [t_0, +\infty)$, we have that :
\begin{equation}\label{thirdderiv}
 \| x^{(3)}(t) \| \leq K( \| \dot{x}(t) \| +  \| \ddot{x}(t) \|).
 \end{equation}
\end{lemma}
\begin{proof}
For  $h > 0 $ we consider the following inequalities :
\begin{align*}
 \| \dot{X}(t+h) - X(t) \|_{1}&= \| F(t+h, X(t+h)) - F(t,X(t)) \|_{1} \\
&\leq \left\| \left( \dot{x}(t+h) - \dot{x}(t) \right) \right\| + \alpha \left\| \dfrac{1}{t+h} \dot{x}(t+h) - \dfrac{1}{t} \dot{x}(t) \right\| \\
& + \left\| \nabla g \left( x(t+h) + \left(\g+\dfrac{\b}{t+h}\right) \dot{x}(t+h) \right) - \nabla g \left( x(t) +\left(\g+ \dfrac{\b}{t}\right) \dot{x}(t) \right) \right\| \\
&\leq \left\| \left( \dot{x}(t+h) - \dot{x}(t) \right) \right\| + \alpha \left\| \dfrac{1}{t+h} \dot{x}(t+h) - \dfrac{1}{t} \dot{x}(t) \right\| \\
&+ L_g \| x(t+h)-x(t)\| +L_g\left\|\left(\g+\dfrac{\b}{t+h}\right) \dot{x}(t+h)  - \left(\g+\dfrac{\b}{t}\right) \dot{x}(t) \right\|.
\end{align*}
Now, dividing by $h > 0$ and taking the limit ${h \To 0}$, it follows that
\begin{align*}
\| \ddot{X}(t) \|_{1} \leq \| \ddot{x}(t) \| + \alpha \left\| \left( \dfrac{1}{t} \dot{x}(t) \right)^{\prime} \right\| + L_g \| \dot{x}(t) \| + L_g \left\| \left( \left(\g+\dfrac{\b}{t}\right) \dot{x}(t) \right)^{\prime} \right\|.
\end{align*}
Consequently,
\begin{align*}
\| x^{(3)}(t) \| \leq \alpha \left\| -\dfrac{1}{t^2} \dot{x}(t) + \dfrac{1}{t} \ddot{x}(t) \right\| + L_g \| \dot{x}(t) \| + L_g \left\| -\dfrac{\b}{t^2} \dot{x}(t) + \left( \g + \dfrac{ \b}{t} \right) \ddot{x}(t) \right\|.
\end{align*}
Finally,
$$ \| x^{(3)}(t) \| \leq \left( L_g + \dfrac{\a+L_g|\b|}{t^2}\right) \| \dot{x}(t) \| + \left( \frac{\a}{t} + L_g \left|\g + \dfrac{\b}{t} \right| \right) \| \ddot{x}(t) \|\le K(\| \dot{x}(t) \|+\| \ddot{x}(t) \|),$$
where $K=\max\left\{\max_{t\ge t_0}\left( L_g + \dfrac{\a+L_g|\b|}{t^2}\right),\max_{t\ge t_0}\left( \frac{\a}{t} + L_g \left|\g + \dfrac{\b}{t} \right|\right)\right\}$.
\end{proof}

\section{Convergence analysis}

\subsection{On a general energy functional associated to the dynamical system \eqref{dysy}}
In order to obtain convergence rates for the function values in the trajectories generated by the dynamical system \eqref{dysy}, we need to introduce an appropriate energy functional which will play the role of a Lyapunov function. The form of such an energy  functional associated to heavy ball system with vanishing damping and its extensions is well known in the literature, see for instance \cite{att-c-p-r-math-pr2018}, \cite{att-p-r-jde2016}, \cite{AD} and \cite{ADR}. However, the novelty of the dynamical system \eqref{dysy}, compared with the extended/perturbed variants of the heavy ball system studied in the above mentioned papers, consists in the fact that in system \eqref{dysy} the perturbation is carried out in the argument of the gradient of the objective function. This seems to be a new approach in the literature, therefore the previously mentioned energy functionals are not suitable for a valuable convergence analysis of the dynamical system \eqref{dysy}.
Hence, let us denote  $\a(t)=\frac{\a}{t}$ and $\b(t)=\g+\frac{\b}{t}$, and assume  that $\argmin g\neq\emptyset.$ Further, let $g^*=\min g=g(x^*),\,x^*\in\argmin g.$
In connection to the dynamical system \eqref{dysy}, we  introduce the general energy  functional
\begin{equation}\label{energyfunction}
\mathcal{E}(t)=a(t)(g(x(t)+\b(t)\dot{x}(t))-g^*)+\frac12\|b(t)(x(t)-x^*)+c(t)\dot{x}(t)\|^2+\frac{d(t)}{2}\|x(t)-x^*\|^2,
\end{equation}
which can be seen as an extension of the energy function studied in \cite{att-c-p-r-math-pr2018} in connection to the heavy ball system with vanishing damping.

Our purpose is to define the non-negative functions $a(t),b(t),c(t),d(t)$ such that $\dot{\mathcal{E}}(t)\le 0$, that is, the function $\mathcal{E}(t)$ is non-increasing after a $t_1\ge t_0.$ Indeed, if $\mathcal{E}(t)$ is non-increasing for $t\ge t_1,$ then
$$a(t)(g(x(t)+\b(t)\dot{x}(t))-g^*)\le \mathcal{E}(t)\le \mathcal{E}(t_1),$$
in other words
$$ g(x(t)+\b(t)\dot{x}(t))-g^*\le \frac{\mathcal{E}(t_1)}{a(t)},\mbox{ for all }t\ge t_1.$$

In what follows we derive  the conditions which must be imposed on the positive functions $a(t),b(t),c(t),d(t)$ in order to obtain $\dot{\mathcal{E}}(t)\le 0$ for every $t\ge t_1.$
We have,
$$\dot{\mathcal{E}}(t)=a'(t)(g(x(t)+\b(t)\dot{x}(t))-g^*)+a(t)\<\n g(x(t)+\b(t)\dot{x}(t)),\b(t)\ddot{x}(t)+(\b'(t)+1)\dot{x}(t)\>+$$
$$\<b'(t)(x(t)-x^*)+(b(t)+c'(t))\dot{x}(t)+c(t)\ddot{x}(t),b(t)(x(t)-x^*)+c(t)\dot{x}(t)\>+$$
$$\frac{d'(t)}{2}\|x(t)-x^*\|^2+d(t)\<\dot{x}(t),x(t)-x^*\>.$$

Now, from \eqref{dysy} we get
$$\ddot{x}(t)=-\a(t)\dot{x}(t)-\n g(x(t)+\b(t)\dot{x}(t)),$$
hence
$$\dot{\mathcal{E}}(t)=a'(t)(g(x(t)+\b(t)\dot{x}(t))-g^*)+$$
$$a(t)\left \<\n g(x(t)+\b(t)\dot{x}(t)),-\b(t)\n g(x(t)+\b(t)\dot{x}(t))+\left(-\b(t)\a(t)+\b'(t)+1\right)\dot{x}(t)\right\>+$$
$$\left\<b'(t)(x(t)-x^*)+\left(b(t)+c'(t)-c(t)\a(t)\right)\dot{x}(t)-c(t)\n g(x(t)+\b(t)\dot{x}(t)),b(t)(x(t)-x^*)+c(t)\dot{x}(t)\right\>+$$
$$\frac{d'(t)}{2}\|x(t)-x^*\|^2+d(t)\<\dot{x}(t),x(t)-x^*\>.$$
Consequently,
\begin{equation}\label{generg1}\dot{\mathcal{E}}(t)=a'(t)(g(x(t)+\b(t)\dot{x}(t))-g^*)+
\end{equation}
$$-a(t)\b(t)\|\n g(x(t)+\b(t)\dot{x}(t))\|^2+\left(-a(t)\a(t)\b(t)+a(t)\b'(t)+a(t)-c^2(t)\right)\<\n g(x(t)+\b(t)\dot{x}(t)),\dot{x}(t)\>+$$
$$\left(b'(t)b(t)+\frac{d'(t)}{2}\right)\|x(t)-x^*\|^2+\left(b^2(t)+b(t)c'(t)+b'(t)c(t)-b(t)c(t)\a(t)+d(t)\right)\<\dot{x}(t),x(t)-x^*\>+$$
$$c(t)\left(b(t)+c'(t)-c(t)\a(t)\right)\|\dot{x}(t)\|^2-b(t)c(t)\<\n g(x(t)+\b(t)\dot{x}(t)),x(t)-x^*\>.$$

But
$$\<\n g(x(t)+\b(t)\dot{x}(t)),x(t)-x^*\>=\<\n g(x(t)+\b(t)\dot{x}(t)),x(t)+\b(t)\dot{x}(t)-x^*\>-\<\n g(x(t)+\b(t)\dot{x}(t)),\b(t)\dot{x}(t)\>$$
and by the convexity of $g$ we have
$$\<\n g(x(t)+\b(t)\dot{x}(t)),x(t)+\b(t)\dot{x}(t)-x^*\>\ge g(x(t)+\b(t)\dot{x}(t))-g^*,$$
hence
$$-b(t)c(t)\<\n g(x(t)+\b(t)\dot{x}(t)),x(t)-x^*\>\le$$
$$ -b(t)c(t)( g(x(t)+\b(t)\dot{x}(t))-g^*)+b(t)c(t)\b(t)\<\n g(x(t)+\b(t)\dot{x}(t)),\dot{x}(t)\>.$$

Therefore, \eqref{generg1} becomes
\begin{equation}\label{generg2}
\dot{\mathcal{E}}(t)\le (a'(t)-b(t)c(t))(g(x(t)+\b(t)\dot{x}(t))-g^*)-a(t)\b(t)\|\n g(x(t)+\b(t)\dot{x}(t))\|^2+
\end{equation}
$$\left(-a(t)\a(t)\b(t)+a(t)\b'(t)+a(t)-c^2(t)+b(t)c(t)\b(t)\right)\<\n g(x(t)+\b(t)\dot{x}(t)),\dot{x}(t)\>+$$
$$\left(b'(t)b(t)+\frac{d'(t)}{2}\right)\|x(t)-x^*\|^2+\left(b^2(t)+b(t)c'(t)+b'(t)c(t)-b(t)c(t)\a(t)+d(t)\right)\<\dot{x}(t),x(t)-x^*\>+$$
$$c(t)\left(b(t)+c'(t)-c(t)\a(t)\right)\|\dot{x}(t)\|^2$$

In order to have $\dot{\mathcal{E}}(t)\le 0$ for all $t\ge t_1,\,t_1\ge t_0,$ one must assume that for all $t\ge t_1$ the following inequalities hold:
\begin{align}\label{cond1}
a'(t)-b(t)c(t)\le 0,&\\ \label{cond2}
-a(t)\b(t)\le 0,&\\ \label{cond3}
-a(t)\a(t)\b(t)+a(t)\b'(t)+a(t)-c^2(t)+b(t)c(t)\b(t)= 0,&\\ \label{cond4}
b'(t)b(t)+\frac{d'(t)}{2}\le 0,&\\ \label{cond5}
b'(t)c(t)+b(t)\left(b(t)+c'(t)-c(t)\a(t)\right)+d(t)= 0,&\\ \label{cond6}
c(t)\left(b(t)+c'(t)-c(t)\a(t)\right)\le 0.&
\end{align}

\begin{remark}\label{conditions} Observe that \eqref{cond2} implies that $\b(t)\ge 0$ for all $t\ge t_1,\,t_1\ge t_0$ and  this shows that in dynamical system \eqref{dysy}, one must have $\b\ge0$ whenever $\g=0.$ Further, \eqref{cond4} is satisfied whenever $b(t)$ and $d(t)$ are constant functions. It is obvious that there exists $t_1$ such that for all $t\ge t_1$ we have  $-\a(t)\b(t)+\b'(t)+1=1-\frac{\a\g}{t}-\frac{\a\b+\b}{t^2}>0,$ hence from \eqref{cond3} we get
$$a(t)=\frac{c^2(t)-b(t)c(t)\b(t)}{-\a(t)\b(t)+\b'(t)+1}.$$

Since \eqref{cond5} and \eqref{cond6} do not depend by $\b(t),$ it seem natural to choose $c(t)$, $b(t)$ and $d(t)$ the same as in case of heavy ball system with vanishing damping (see \cite{att-c-p-r-math-pr2018}), that is,
$c(t)=t,\, b(t)=b\in(0,\a-1]$ and $d(t)=b(\a-1-b),$ for all $t\ge t_0,$ provided $\a>1.$ Now, an easy computation shows that in this case
$$a(t)=t^2+\g(\a-b)t+\b+(\b+\a\g^2)(\a-b)+$$
$$\frac{(\a\b\g+(\g\b+2\a\b\g+\a^2\g^3)(\a-b))t+(\a\b+\b)(\b+(\b+\a\g^2)(\a-b))}{t^2-\a\g t-(\a\b+\b)},$$
hence \eqref{cond1} is satisfied whenever $b>2,$ which implies that $\a>3.$

However, if $\g=0$ then
$$a(t)=t^2+\b+\b(\a-b)+\frac{((\a\b+\b)(\b+\b(\a-b))}{t^2-(\a\b+\b)},$$
hence \eqref{cond1} holds also for $b=2$ and $\a=3.$
\end{remark}

\subsection{Error estimates for  the values}
In this section we obtain convergence rate of order $\mathcal{O}(1/t^2),\,t\To+\infty$ for the difference $g\left(x(t)+\left(\g+\frac{\b}{t}\right)\dot{x}(t)\right)-g^*$ where $g^*=g(x^*)=\min g,\,x^*\in\argmin g\neq\emptyset.$ From here we are able to show that $g(x(t))-g^*$ also has a convergence rate of order $\mathcal{O}(1/t^2),\,t\To+\infty.$
However, just as in the case of heavy ball system with vanishing damping, in order to obtain these rates, it is necessary to assume $\a\ge 3$ in our system \eqref{dysy}.
We have the following result.

\begin{theorem}\label{convexconvergence}
Let $x$ be the unique strong global solution of \eqref{dysy} and assume that $\argmin g\neq\emptyset.$
Assume further that $\a> 3$, $\g>0$ and $\b\in\R$.

Then, there exists $K\ge 0$ and $t_1\ge t_0$ such that
\begin{equation}\label{quadrate}
g\left(x(t)+\left(\g+\frac{\b}{t}\right)\dot{x}(t)\right)-\min g\le \frac{K}{t^2},\,\mbox{for all }t\ge t_1.
\end{equation}
Further,
\begin{equation}\label{ergrate}
\int_{t_0}^{+\infty}t\left(g\left(x(t)+\left(\g+\frac{\b}{t}\right)\dot{x}(t)\right)-\min g\right)dt<+\infty,
\end{equation}
and
\begin{equation}\label{ergrategradient}
\int_{t_0}^{+\infty}t^2\left\|\n g\left(x(t)+\left(\g+\frac{\b}{t}\right)\dot{x}(t)\right)\right\|^2dt<+\infty.
\end{equation}
\end{theorem}
\begin{proof}
 Let $\min g=g^*=g(x^*),\,x^*\in\argmin g.$ Consider the energy functional \eqref{energyfunction} 
 with
 $$b(t)=\a-1>2,\,c(t)=t,\,d(t)=0.$$
 According to Remark \ref{conditions} we have
 $$a(t)=t^2+\g t+2\b+\a\g^2+\frac{(3\a\b\g+\a^2\g^3+\b\g)t+\b(\a+1)(2\b+\a\g^2)}{t^2-\a\g t-\b(\a+1)}$$ and
 the conditions  \eqref{cond3}-\eqref{cond6} are satisfied with equality for every $t\ge t_0$.

Obviously,  if $t$ is big enough on has that $a(t)>0$ and since $\g>0$ it holds that $\g+\frac{\b}{t}>0$ for $t$ big enough, even if $\b<0$. Hence, there exists $t'\ge t_0$ such \eqref{cond2} is satisfied for every $t\ge t'.$

Now, $a'(t)-b(t)c(t)=(3-\a)t+\g+\mathcal{O}\left(\frac{1}{t^2}\right)$ and by taking into account that $\a>3$ we obtain that there exists $t''\ge t_0$ such that \eqref{cond1} holds for every $t\ge t''.$

Let $t'''=\max(t',t'').$ We conclude that, $\dot{\mathcal{E}}(t)\le0 $ for all $t\ge t''',$ hence $\mathcal{E}(t)$ is nonincreasing, i.e.
$$a(t)\left(g\left(x(t)+\left(\g+\frac{\b}{t}\right)\dot{x}(t)\right)-\min g\right)\le \mathcal{E}(t)\le \mathcal{E}(t'''),\mbox{ for all }t\ge t'''.$$

 But, obviously there exists $t_1\ge t'''$ such that $a(t)\ge t^2$ for all $t\ge t_1$ and by denoting $K=\mathcal{E}(t''')$ we obtain

 $$g\left(x(t)+\left(\g+\frac{\b}{t}\right)\dot{x}(t)\right)-\min g\le \frac{K}{t^2},\mbox{ for all }t\ge t_1.$$

 Next we prove \eqref{ergrate}. Since $a'(t)-b(t)c(t)=(3-\a)t+\g+\mathcal{O}\left(\frac{1}{t^2}\right)$ there exist $t_2\ge t_1$ such that
 $$a'(t)-b(t)c(t)\le \frac{(3-\a)t}{2},\mbox{ for all }t\ge t_2.$$

 Hence,
 $$\dot{\mathcal{E}}(t)\le \frac{(3-\a)t}{2} \left(g\left(x(t)+\left(\g+\frac{\b}{t}\right)\dot{x}(t)\right)-\min g\right),\mbox{ for all }t\ge t_2$$
 and by integrating from $t_2$  to $T> t_2$  we get
 $$\int_{t_2}^{T}t\left(g\left(x(t)+\left(\g+\frac{\b}{t}\right)\dot{x}(t)\right)-\min g\right)dt\le\frac{2}{3-\a}(\mathcal{E}(T)-\mathcal{E}(t_2)).$$
Now, by letting $T\To+\infty$ and taking into account that $\mathcal{E}$ is nonincreasing, the conclusion follows.

For proving \eqref{ergrategradient} observe that there exists $t_3\ge t_1$ such that $-a(t)\left(\g+\frac{\b}{t}\right)\le-\frac{\g}{2} t^2$ for all $t\ge t_3.$
Consequently $$\dot{\mathcal{E}}(t)\le -\frac{\g}{2} t^2 \left\|\n g\left(x(t)+\left(\g+\frac{\b}{t}\right)\dot{x}(t)\right)\right\|^2,\mbox{ for all }t\ge t_3.$$

Integrating from $t_3$ to $T>t_3$ and letting $T\To+\infty$ we obtain the desired conclusion.
 \end{proof}
\begin{remark}\label{casegnul1}
Note that according to Remark \ref{conditions} the condition \eqref{cond2} holds even if $\g=0$ and $\b\ge0.$
Consequently, even in the case  $\a>3,\,\g=0,\,\b\ge0$ the conclusions \eqref{quadrate} and \eqref{ergrate} in Theorem \ref{convexconvergence} hold.
However, if $\b=0$ then we cannot obtain \eqref{ergrategradient}. Nevertheless, if $\b>0$ then \eqref{ergrategradient} becomes:
$$\int_{t_0}^{+\infty}t\left\|\n g\left(x(t)+\frac{\b}{t}\dot{x}(t)\right)\right\|^2dt<+\infty.$$

\end{remark}
\begin{remark}\label{casegnul} Concerning the case $\a=3$, note that \eqref{cond2} is  satisfied if one assumes that $\g=0$ but $\b\ge 0.$ Moreover, in this case \eqref{cond1} is also satisfied since, one has
$a'(t)-b(t)c(t)=- \frac{4\b^2(\a+1)}{(t^2-\b(\a+1))^2}\le 0.$ Hence, also in the case $\a=3,\,\g=0,\,\b\ge0$  \eqref{quadrate} in the conclusion of Theorem \ref{convexconvergence} holds.

Moreover,  if we assume $\b>0$ then \eqref{ergrategradient} becomes:
$$\int_{t_0}^{+\infty}t\left\|\n g\left(x(t)+\frac{\b}{t}\dot{x}(t)\right)\right\|^2dt<+\infty.$$
\end{remark}

Next we show that also the  error $g\left(x(t)\right)-\min g$ is of order $\mathcal{O}(1/t^2).$ For obtaining this result we need the Descent Lemma, see \cite{Nest}.

\begin{lemma}\label{desc} Let $g:\R^m\To\R$ be a Fr\`echet differentiable function with $L_g$ Lipschitz continuous gradient. Then one has
$$g(y)\le g(x)+\<\n g(x),y-x\>+\frac{L_g}{2}\|y-x\|^2,\,\forall x,y\in\R^m.$$
\end{lemma}

\begin{theorem}\label{convconverg}
Let $x$ be the unique strong global solution of \eqref{dysy} and assume that $\argmin g\neq\emptyset.$
If $\a> 3,\,\g>0$ and $\b\in\R$, then $x$ is bounded and there exists $K\ge 0$ and $t_1\ge t_0$ such that
\begin{equation}\label{velocityrate}
\|\dot{x}(t)\|\le\frac{K}{t},\mbox{ for all }t\ge t_1.
\end{equation}
Further,
\begin{equation}\label{velocityergrate}
\int_{t_0}^{+\infty}t\|\dot{x}(t)\|^2 dt\le+\infty
\end{equation}
and there exists $K_1\ge 0$ and $t_2\ge t$ such that
\begin{equation}\label{valuerate}
g\left(x(t)\right)-\min g\le\frac{K_1}{t^2},\mbox{ for all }t\ge t_2.
\end{equation}
\end{theorem}
\begin{proof}
For $x^*\in\argmin g$ let $g^*=g(x^*)=\min g$ and  consider the energy function \eqref{energyfunction} 
 with
 $b(t)=b,$ where $2<b<\a-1$, $c(t)=t$, $d(t)=b(\a-1-b)>0$ and
 $$a(t)=\frac{c^2(t)-b(t)c(t)\left(\g+\frac{\b}{t}\right)}{1-\frac{\b}{t^2}-\frac{\a}{t}\left(\g+\frac{\b}{t}\right)}=\frac{(t^2-b \g t-b\b)t^2}{t^2-\a\g t-\b(\a+1)}=\left(1+\frac{(\a-b)\g t-\b(\a+1-b)}{t^2-\a\g t-\b(\a+1)}\right)t^2.$$
 According to Remark \ref{conditions} there exists $t_1\ge t_0$ such that the conditions \eqref{cond1}-\eqref{cond6} are satisfied.

From the definition of $\mathcal{E}$ one has
$$\|b(t)(x(t)-x^*)+c(t)\dot{x}(t)\|\le\sqrt{2\mathcal{E}(t)},$$
and
$$\|x(t)-x^*\|\le \sqrt{\frac{2}{d(t)}\mathcal{E}(t)}$$
that is
$$\|b(x(t)-x^*)+t\dot{x}(t)\|\le\sqrt{2\mathcal{E}(t)},$$
and
$$\|x(t)-x^*\|\le \sqrt{\frac{2}{b(\a-1-b)}\mathcal{E}(t)}.$$
By using the fact that $\mathcal{E}$ nonincreasing on an interval $[t_1,+\infty)$ the latter inequality assures that $x$ is bounded.

Now, by using the inequality $\|X-Y\|\ge\|X\|-\|Y\|$ we get $$\|b(x(t)-x^*)+t\dot{x}(t)\|\ge t\|\dot{x}(t)\|-\|b(x(t)-x^*)\|,$$ hence for all $t\ge t_1$ one has
$$t\|\dot{x}(t)\|\le \sqrt{2\mathcal{E}(t)}+\|b(x(t)-x^*)\|\le\left(1+\sqrt{\frac{1}{\a-1-b}}\right)\sqrt{2\mathcal{E}(t)}\le K,$$
where $K=\left(1+\sqrt{\frac{1}{\a-1-b}}\right)\sqrt{2\mathcal{E}(t_1)}.$

Further, \eqref{cond6} becomes $(b+1-\a)t<0,$ hence for all $t\ge t_1$ one has
$$\dot{\mathcal{E}}(t)\le(b+1-\a)t\|\dot{x}(t)\|^2.$$

By integrating from $t_1$ to $T>t_1$ one gets
$$\int_{t_1}^Tt\|\dot{x}(t)\|^2dt\le\frac{1}{\a-1-b}(\mathcal{E}(t_1)-\mathcal{E}(T))\le \frac{1}{\a-1-b}\mathcal{E}(t_1).$$
By letting $T\To+\infty$ we obtain
$$\int_{t_0}^{+\infty}t\|\dot{x}(t)\|^2 dt\le+\infty.$$

Now, by using Lemma \ref{desc} with $y=x(t)$ and $x=x(t)+\left(\g+\frac{\b}{t}\right)\dot{x}(t)$ we obtain
\begin{equation}\label{e}
g(x(t))- g\left(x(t)+\left(\g+\frac{\b}{t}\right)\dot{x}(t)\right)\le
\end{equation}
$$\left\<\n g\left(x(t)+\left(\g+\frac{\b}{t}\right)\dot{x}(t)\right),-\left(\g+\frac{\b}{t}\right)\dot{x}(t)\right\>+\frac{L_g}{2}\left\|\left(\g+\frac{\b}{t}\right)\dot{x}(t)\right\|^2\le$$
$$\frac{|\g t+\b|}{t}\left\|\n g\left(x(t)+\left(\g+\frac{\b}{t}\right)\dot{x}(t)\right)\right\|\left\|\dot{x}(t)\right\|+\frac{L_g}{2}\left(\g+\frac{\b}{t}\right)^2\|\dot{x}(t)\|^2.$$

Now according to \eqref{ergrategradient} there exists $t'\ge t_0$ and $K'>0$ such that
$$\left\|\n g\left(x(t)+\left(\g+\frac{\b}{t}\right)\dot{x}(t)\right)\right\|\le\frac{K'}{t},\mbox{ for all }t\ge t'.$$
Further \eqref{velocityrate} assures that there exists $t_1\ge t_0$ and $K>0$ such  that
$$\|\dot{x}(t)\|\le\frac{K}{t},\mbox{ for all }t\ge t_1.$$
Obviously, $$\left|\g+\frac{\b}{t}\right|\le \max\left(\g,\g+\frac{\b}{t_0}\right),\mbox{ for all }t\ge t_0,$$
hence \eqref{e} assures that there exists $K_1'>0$ and $t_2'\ge t_0$ such that
\begin{equation}\label{decay}
g(x(t))- g\left(x(t)+\left(\g+\frac{\b}{t}\right)\dot{x}(t)\right)\le\frac{K_1'}{t^2},\mbox{ for all }t\ge t_2'.
\end{equation}
Now, by adding \eqref{quadrate} and \eqref{decay} we get that there exists $K_1>0$ and $t_2\ge t_0$ such that
$$\left(g\left(x(t)+\left(\g+\frac{\b}{t}\right)\dot{x}(t)\right)-\min g\right)+\left(g(x(t))- g\left(x(t)+\left(\g+\frac{\b}{t}\right)\dot{x}(t)\right)\right)\le\frac{K_1}{t^2},\mbox{ for all }t\ge t_2$$ that is

$$g(x(t))-\min g\le \frac{K_1}{t^2},\mbox{ for all }t\ge t_2.$$
\end{proof}

\begin{remark} Note that if we assume that $\b\ge0$ then one can allow $\g=0$ in the hypotheses of Theorem \ref{convconverg}, since in this case condition \eqref{cond2} is satisfied and the conclusions of Theorem \ref{convconverg} hold.
\end{remark}

\begin{remark}\label{L1deriv} Note that \eqref{ergrategradient}, which holds whenever $\a>3$ and $\g>0$, assures that
$$t\left\|\n g\left(x(t)+\left(\g+\frac{\b}{t}\right)\dot{x}(t)\right)\right\|\in L^2([t_0,+\infty),\R).$$
Further, \eqref{velocityergrate} assures that
\begin{equation}\label{firstorderderivlp}
\sqrt{t}\|\dot{x}(t)\|\in L^2([t_0,+\infty),\R).
\end{equation}

Consequently, the system \eqref{dysy} leads to
$$\|t\ddot{x}(t)\|=\left\|\a\dot{x}(t)+t\n g\left(x(t)+\left(\g+\frac{\b}{t}\right)\dot{x}(t)\right)\right\|\le \a\|\dot{x}(t)\|+t\left\|\n g\left(x(t)+\left(\g+\frac{\b}{t}\right)\dot{x}(t)\right)\right\|$$
hence,
\begin{equation}\label{seconderivlp}
t\|\ddot{x}(t)\|\in L^2([t_0,+\infty),\R).
\end{equation}
Now, from \eqref{thirdderiv} we have
 $$\sqrt{t}\| x^{(3)}(t) \| \leq K(\sqrt{t} \| \dot{x}(t) \| + \sqrt{t} \| \ddot{x}(t) \|)$$
 for some $K>0$ and for almost every $t\ge t_0$, which combined with \eqref{firstorderderivlp} and \eqref{seconderivlp} gives
\begin{equation}\label{thirdderivlp}
\sqrt{t}\| x^{(3)}(t) \|\in L^2([t_0,+\infty),\R).
\end{equation}
\end{remark}

\begin{remark} Notice that \eqref{firstorderderivlp}, \eqref{seconderivlp} and \eqref{thirdderivlp} assure in particular that
\begin{equation}\label{zerolimit}
\lim_{t\To+\infty}\|\dot{x}(t)\|=\lim_{t\To+\infty}\|\ddot{x}(t)\|=\lim_{t\To+\infty}\|{x}^{(3)}(t)\|=0.
\end{equation}
\end{remark}

\begin{remark}\label{zerolimitgnull} In the case $\g=0$ and $\b\ge0$ according to Remark \ref{casegnul1} one has $$\sqrt{t}\left\|\n g\left(x(t)+\left(\g+\frac{\b}{t}\right)\dot{x}(t)\right)\right\|\in L^2([t_0,+\infty),\R).$$ Hence, as in Remark \ref{L1deriv} we derive that
$$\sqrt{t}\|\ddot{x}(t)\|,\,\sqrt{t}\| x^{(3)}(t) \|\in L^2([t_0,+\infty),\R)$$
and consequently \eqref{zerolimit} holds.
\end{remark}

\subsection{The convergence of the generated trajectories}

In this section we show convergence of the generated trajectories to a minimum point of the objective $g.$ The main tool that we use in order to attain this goal will be the following continuous version of Opial Lemma, (see \cite{Bruck, Opial, att-c-p-r-math-pr2018}.

\begin{lemma}\label{Opial}
Let $H$ to be a separable Hilbert space and $S \subseteq H$, with $S \neq \emptyset$. Further, consider $x : [t_0, +\infty) \to H$ a given map. Suppose that the following conditions are satisfied :
\begin{align*}
& (i) \quad \forall z \in S, \, \exists \lim\limits_{t \To +\infty} \| x(t) - z \| \\
& (ii) \quad \text{every weak sequentially limit point of } x(t) \text{ belongs to the set }S.
\end{align*}
Then, $x(t)$ converges weakly to a point of $S$ as $t \to + \infty$.
\end{lemma}

\begin{remark}
In the setting of the proof concerning the convergence of the generated trajectories, we consider the set $S$ to be $argmin(g)$. Moreover, the Hilbert space $H$ is  $\mathbb{R}^m$, and this implies that we actually deduce strong convergence of a strong global solution of the dynamical system \eqref{dysy} to a minimum of the objective function $g.$
\end{remark}

Consider now the set of limit points of the trajectory $x$, that is
$$\omega (x) =\{ \ol x\in\R^m: \exists (t_n)_{n\in\N}\subseteq\R,\, t_n\To+\infty,\,n\To+\infty \text{ such that } x(t_n) \To \ol{x},\, n\To+\infty\}.$$

We show that $\omega(x)\subseteq \argmin g.$  We emphasize that since $g$ is convex one has
 $$\argmin g=\crit g:=\{\ol x\in\R^m:\n g(\ol x)=0.\}.$$
 We have the following result.
 \begin{lemma}\label{limitpoint} Let $x$ be the unique strong global solution of \eqref{dysy} and assume that $\argmin g\neq\emptyset.$
If $\a> 3,\,\g>0$ and $\b\in\R$, then the following assumptions hold.
\begin{itemize}
\item[(i)] $\omega(x)=\omega\left(x(\cdot)+\left(\g+\frac{\b}{t}\right)\dot{x}(\cdot)\right)$;
\item[(ii)] $\omega(x)\subseteq\argmin g.$
\end{itemize}
  \end{lemma}
 \begin{proof}
 Indeed \eqref{zerolimit} assures  that $\lim_{t \To+\infty} \left( \gamma + \dfrac{\beta}{t} \right) \dot{x}(t) = 0$, which immediately  proves (i).

For proving (ii) consider $\ol x\in \omega(x).$ Then, there exists $(t_n)_{n\in\N}\subseteq\R,\, t_n\To+\infty,\,n\To+\infty$ such that
 $\lim_{n\To+\infty}x(t_n) = \ol{x}.$ Now, since $\n g$ is continuous and
 $\lim_{n\To+\infty}\left(x(t_n)+\left(\g+\frac{\b}{t_n}\right)\dot{x}(t_n)\right)=\ol x$
 one has
$$\lim_{n\To+\infty}\n g\left(x(t_n)+\left(\g+\frac{\b}{t_n}\right)\dot{x}(t_n)\right)=\n g(\ol x).$$

Further, according to \eqref{zerolimit}
$$\lim_{n\To+\infty}\left(\ddot{x}(t_n)+\frac{\a}{t_n}\dot{x}(t_n)\right)=0.$$
Now, the system \eqref{dysy} gives
$$0=\lim_{n\To+\infty}\left(\ddot{x}(t_n)+\frac{\a}{t_n}\dot{x}(t_n)+\n g\left(x(t_n)+\left(\g+\frac{\b}{t_n}\right)\dot{x}(t_n)\right)\right)=\n g(\ol x)$$
that is $\ol x\in\argmin g$ and this proves (ii).
\end{proof}
\begin{remark}\label{Opialgnull} Obviously, according to Remark \ref{zerolimitgnull} the conclusion of Lemma \ref{limitpoint} remains valid also in the case  $\a>3,\,\g=0$ and $\b\ge0.$
Note that (ii) from the conclusion of Lemma \ref{limitpoint} is actually the condition (ii) from Lemma \ref{Opial}. For proving (i) from Lemma \ref{Opial}, that is, the limit $\lim_{t\To+\infty}\|x(t)-x^*\|$ exists for every $x^*\in\argmin g,$ we need the following result from \cite{att-c-p-r-math-pr2018}.
\end{remark}

\begin{lemma}\label{existenceof limit}(Lemma A.4.\, \cite{att-c-p-r-math-pr2018}) Let $t_0>0$, and let $w:[t_0,+\infty)\To\R$ be a continuously differentiable function which is bounded from below. Assume that
$$ t\ddot{w}(t) + \a \dot{w}(t) \le G(t)$$
for some $\a > 1,$  almost every $t >t_0$, and some nonnegative function $G\in L^1(t_0,+\infty).$ Then, the positive part $[\dot{w}]_+$ of $\dot{w}$ belongs
to $L^1(t_0,+\infty)$ and limit $\lim_{t\To+\infty} w(t)$ exists.
\end{lemma}
Now we can prove the following.

\begin{lemma}\label{limexist} Let $x$ be the unique strong global solution of \eqref{dysy} and assume that $\argmin g\neq\emptyset.$
If $\a> 3,\,\g>0$ and $\b\in\R$, then for every $x^*\in\argmin g$ there exists the limit
$$\lim_{t\To+\infty}\|x(t)-x^*\|.$$
\end{lemma}
\begin{proof}
Let $ x^*\in \argmin g$ and define the function $h_{x^*}:[t_0,+\infty)\To\R,\,h_{x^*}(t)= \dfrac{1}{2} \| x(t) - x^* \|^2$. Using the chain rule with respect to the differentiation of $h_{x^*}$, we obtain that
\begin{align*}
\dot{h}_{x^*}(t) = \< \dot{x}(t), x(t)-x^* \>
\end{align*}
and that
\begin{align*}
\ddot{h}_{x^*}(t) = \<\ddot{x}(t), x(t)-x^* \> + \| \dot{x}(t) \|^2.
\end{align*}
Let us denote $g^{\ast} = g(x^{\ast}) = min(g)$. Using the dynamical system (\ref{dysy}), one has that
\begin{align*}
\ddot{h}_{x^*}(t) + \dfrac{\alpha}{t} \dot{h}_{x^*}(t) = -\left\< \nabla g \left( x(t) + \left( \gamma + \dfrac{\beta}{t} \right) \dot{x}(t) \right), x(t) - x^{\ast} \right\> + \| \dot{x}(t) \|^2.
\end{align*}
This means that
\begin{equation}\label{use}
\ddot{h}_{x^*}(t) + \dfrac{\alpha}{t} \dot{h}_{x^*}(t) =- \left\< \nabla g \left( x(t) + \left( \gamma + \dfrac{\beta}{t} \right) \dot{x}(t) \right), x(t) + \left( \gamma + \dfrac{\beta}{t} \right) \dot{x}(t)- x^* \right\> + \| \dot{x}(t) \|^2
\end{equation}
$$
 + \left\< \nabla g \left( x(t) + \left( \gamma + \dfrac{\beta}{t} \right) \dot{x}(t) \right), \left( \gamma + \dfrac{\beta}{t} \right) \dot{x}(t) \right\>.
$$
By the convexity of the mapping $g$ and taking into account that $g^*=g(x^*)=\min g$, we obtain
$$- \left\< \nabla g \left( x(t) + \left( \gamma + \dfrac{\beta}{t} \right) \dot{x}(t) \right), x(t) + \left( \gamma + \dfrac{\beta}{t} \right) \dot{x}(t)- x^* \right\>\le g^* - g \left( x(t) + \left( \gamma + \dfrac{\beta}{t} \right) \dot{x}(t) \right)\le 0,$$
hence  by using \eqref{dysy} the inequality \eqref{use} becomes
\begin{align*}
\ddot{h}_{x^*}(t) + \dfrac{\alpha}{t} \dot{h}_{x^*}(t) \leq g^* - g \left( x(t) + \left( \gamma + \dfrac{\beta}{t} \right) \dot{x}(t) \right) + \left( 1 - \dfrac{\alpha}{t} \left( \gamma + \dfrac{\beta}{t} \right) \right) \| \dot{x}(t) \|^2 - \left( \gamma + \dfrac{\beta}{t} \right) \< \dot{x}(t), \ddot{x}(t) \>.
\end{align*}

Consequently, one has
\begin{equation}\label{use1}
t\ddot{h}_{x^*}(t) + \a \dot{h}_{x^*}(t) \leq  \left( t - \a \left( \gamma + \dfrac{\beta}{t} \right) \right) \| \dot{x}(t) \|^2 - \left( \g t + \b\right) \< \dot{x}(t), \ddot{x}(t) \>.
\end{equation}

Now, according to \eqref{firstorderderivlp} one has
$$t\|\dot{x}(t)\|^2\in L^1([t_0,+\infty),\R).$$
Consequently,
\begin{equation}\label{der1}
\left( t - \a \left( \gamma + \dfrac{\beta}{t} \right) \right) \| \dot{x}(t) \|^2\in L^1([t_0,+\infty),\R).
\end{equation}

Further,
$$- \left( \g t + \b\right) \< \dot{x}(t), \ddot{x}(t) \>\le \frac12\|\dot{x}(t)\|^2+\frac12\left( \g t + \b\right)^2\|\ddot{x}(t)\|^2$$
and  \eqref{seconderivlp} assures that
$$t^2\|\ddot{x}(t)\|^2\in L^1([t_0,+\infty),\R),$$
hence
\begin{equation}\label{der2}
- \left( \g t + \b\right) \< \dot{x}(t), \ddot{x}(t) \>\in L^1([t_0,+\infty),\R).
\end{equation}

According to \eqref{der1} and \eqref{der2} we get that the function
$$G(t)= \left( t - \a \left( \gamma + \dfrac{\beta}{t} \right) \right) \| \dot{x}(t) \|^2 +\frac12\|\dot{x}(t)\|^2+\frac12\left( \g t + \b\right)^2\|\ddot{x}(t)\|^2\in L^1([t_0,+\infty),\R).$$  Moreover, it is obvious that there exists $t_1\ge t_0$ such that $G$ is non negative for every $t\ge t_1.$ From \eqref{use1}
one has
\begin{equation}\label{use2}
t\ddot{h}_{x^*}(t) + \a \dot{h}_{x^*}(t) \leq G(t),\mbox{ for every }t\ge t_0,
\end{equation}
hence Lemma \ref{existenceof limit} leads to the existence of the limit
$$\lim_{t\To+\infty}h_{x^*}(t)$$
and consequently the limit
$$\lim_{t\To+\infty}\|x(t)-x^*\|$$
also exists.
\end{proof}

Now, we present the main result of this subsection regarding the convergence of the solution of the dynamical system (\ref{dysy}) as $t \to + \infty$.

\begin{theorem}\label{convergencetraj}
Let $x$ be the unique strong global solution of the dynamical system (\ref{dysy}) and assume that $\argmin g\neq\emptyset.$
If $\a> 3,\,\g>0$ and $\b\in\R$, then $x(t)$ converges to a point in $\argmin g$ as $t \To +\infty$.
\end{theorem}
\begin{proof}
 Taking into account that the conclusion (ii) in Lemma \ref{limitpoint} and the conclusion of Lemma \ref{limexist} are exactly the conditions in the hypothesis of Lemma \ref{Opial} with $S=\argmin g,$ the conclusion of the theorem is straightforward.
\end{proof}

\begin{remark} As it was expected, Theorem \ref{convergencetraj} remains true if in its hypothesis we assume only that $\a>3,\, \g=0$ and $\b\ge 0.$ Indeed, note that under these assumptions the conclusion of Lemma \ref{limexist} holds, since $G$ from its proof becomes
$$G(t)= \left( t -\dfrac{\a\b}{t}  \right) \| \dot{x}(t) \|^2 +\frac12\|\dot{x}(t)\|^2+\frac12 \b^2\|\ddot{x}(t)\|^2$$
and according to Remark \ref{zerolimitgnull} $G(t)\in L^1([t_0,+\infty),\R).$  Moreover, it is obvious that there exists $t_1\ge t_0$ such that $G$ is non negative for every $t\ge t_1.$

This fact combined with Remark \ref{Opialgnull} lead to the desired conclusion.
\end{remark}

\section{Conclusions}

In this paper we study a second order dynamical system which can be viewed as an extension of the heavy ball system with vanishing damping. This dynamical system is actually a perturbed version of the heavy ball system with vanishing damping, but the perturbation is made in the argument of the gradient of the objective function. Numerical experiments show that this perturbation brings a smoothing effect in the behaviour of the energy error $g(x(t))-\min g$ and also in the behaviour of the absolute error $\|x(t)-x^*\|$, where $x(t)$ is a generated trajectory by our dynamical system and $x^*$ is a minimum of the objective $g.$ Another novelty of our system that via explicit discretization leads to inertial algorithms. A related future research is the convergence analysis of  algorithm \eqref{alggen}, since this algorithm contains as particular case the famous Nesterov algorithm. However, since Algorithm \eqref{alggen} may allow  different inertial steps, its area of applicability can be considerable.

We have shown existence and uniqueness of the trajectories generated by our dynamical system even for the case the objective function $g$ is non-convex. Further, we treated the cases when the energy error $g(x(t))-\min g$ is of order $\mathcal{O}\left(\frac{1}{t^2}\right)$ and we obtained the convergence of a generated trajectory to a minimum of the objective function $g.$
Another related research is the convergence analysis of the generated trajectories in the case when the objective function  $g$ is possible non-convex. This would be a novelty in the literature even  for the case $\a>3,\,\g=\b=0$, that is, for the case of the heavy ball system with vanishing damping.

We underline that the dynamical system \eqref{dysy} can easily be extended to proximal-gradient dynamical systems (see \cite{BCL,BCL-MM} and the references therein). Indeed, let $f:\R^m\To\ol\R$ be a  proper convex and lower semicontinuous function and let $g:\R^m\To\R$ be a possible non-convex smooth function with $L_g$ Lipschitz continuous gradient.

We recall that the proximal point operator of the convex function $\l  f$ is defined as
\begin{equation*}  \prox\nolimits_{\l  f} : {\R^m} \rightarrow {\R^m}, \quad \prox\nolimits_{\l f}(x)=\argmin_{y\in {\R^m}}\left \{f(y)+\frac{1}{2\l}\|y-x\|^2\right\}.
\end{equation*}

Consider the optimization problem
$$\inf_{x\in\R^m}f(x)+g(x).$$

One can associate to this optimization problem the following second order proximal-gradient dynamical system.

\begin{equation}\label{fist}\left\{\begin{array}{lll}
\ds \ddot{x}(t)+\left(\g+\frac{\a+\b}{t}\right)\dot{x}(t)+x(t)=\prox\nolimits_{\l f}\left(\left(x(t)+\left(\g+\frac{\b}{t}\right)\dot{x}(t)\right)-\l\n g\left(x(t)+\left(\g+\frac{\b}{t}\right)\dot{x}(t)\right)\right),\\
\\
\ds x(t_0)=u_0,\,\dot{x}(t_0)=v_0,
\end{array}\right.
\end{equation}
where $\a>0,\,\g>0,\,\l>0,\,\b\in\R$  and $t_0>0.$

Obviously, when $f\equiv 0$ and $\l=1$ then \eqref{fist} becomes the dynamical system \eqref{dysy} studied in the present paper.

The discretization of the the dynamical system \eqref{fist} leads to proximal-gradient inertial algorithms, similar to the modified FISTA algorithm studied by Chambolle and Dossal in \cite{ch-do2015}, (see also \cite{bete09}).

Indeed, the explicit discretization of \eqref{dysy} with respect to the time variable $t$, with step size $h_n$  and initial points $x_0:=u_0,\,x_1:=v_0$ yields the iterative scheme
$$\frac{x_{n+1} - 2x_n + x_{n-1}}{h_n^2} + \left(\g+\frac{\a+\b}{n h_n}\right) \frac{x_{n} - x_{n-1}}{h_n} + x_{n+1} =$$
 $$\prox\nolimits_{\l f}\left( x_n+\left(\g+\frac{\b}{n h_n}\right) \frac{x_{n} - x_{n-1}}{h_n} -\l \nabla g\left(x_n+\left(\g+\frac{\b}{n h_n}\right) \frac{x_{n} - x_{n-1}}{h_n}\right)\right)  \ \forall n \geq 1.$$
 
 The latter can be expressed as
 
 $$(1+h_n^2)x_{n+1}=x_n+\left(1-\g h_n-\frac{\a+\b}{n}\right)(x_n-x_{n-1})+$$
 $$h_n^2 \prox\nolimits_{\l f}\left( x_n+\left(\g+\frac{\b}{n h_n}\right) \frac{x_{n} - x_{n-1}}{h_n} -\l \nabla g\left(x_n+\left(\g+\frac{\b}{n h_n}\right) \frac{x_{n} - x_{n-1}}{h_n}\right)\right).$$
 
 Consequently, the dynamical system \eqref{fist} leads to the algorithm:
For $x_0,\,x_{1}\in\R^m$ consider
\begin{equation}\label{fist1}\left\{\begin{array}{llll}
\ds y_n=x_n+\left(1-\g h_n-\frac{\a+\b}{n}\right)(x_n-x_{n-1}),\\
\\
\ds z_n=x_n+\left(\frac{\g}{h_n}+\frac{\b}{n h_n^2}\right) (x_{n} - x_{n-1}),\\
\\
\ds x_{n+1}=\frac{1}{1+h_n^2}y_n+\frac{h_n^2}{1+h_n^2}\prox\nolimits_{\l f}(z_n-\l\n g(z_n)),
\end{array}\right.
\end{equation}
where $\a\ge 0,\,\g\ge 0$ and $\b\in\R.$  

Now, the simplest case is obtained by taking in \eqref{fist1} constant step size $h_n\equiv 1$. By denoting $\a+\b$ still with $\a\in\R$, the Algorithm
\eqref{fist1} becomes:
For $x_0,x_{1}\in\R^m$ and for every $n\ge 1$ consider
\begin{equation}\label{fist2}\left\{\begin{array}{llll}
\ds y_n=x_n+\frac{(1-\g)n-\a}{n}(x_n-x_{n-1}),\\
\\
\ds z_n=x_n+\frac{\g n+\b}{n} (x_{n} - x_{n-1}),\\
\\
\ds x_{n+1}=\frac{1}{2}y_n+\frac{1}{2}\prox\nolimits_{\l f}(z_n-\l\n g(z_n)),
\end{array}\right.
\end{equation}
where $\a\ge \b,\,\g\ge 0$ and $\b\in\R.$  
The convergence of the sequences generated by Algorithm \eqref{fist2} to a critical point of the objective function $f+g$  would open the gate for the study of FISTA type algorithms with nonidentical inertial terms.


\begin{thebibliography}{99}
\bibitem{AD} J.F. Aujol, Ch. Dossal, {\it Optimal rate of convergence of an ODE associated
to the Fast Gradient Descent schemes for $b > 0$}, HAL preprint, https://hal.inria.fr/hal-01547251v2/document
\bibitem{ADR} J.F. Aujol, Ch. Dossal, A. Rondepierre, {\it Optimal convergence rates for Nesterov acceleration}, arxiv.org/abs/1805.05719
\bibitem{alvarez-attouch2001} F. Alvarez, H. Attouch, {\it An inertial proximal method for maximal monotone operators via discretization
of a nonlinear oscillator with damping}, Set-Valued Analysis, 9, 3-11, 2001
\bibitem{AAD} V. Apidopoulos, J.F. Aujol, Ch. Dossal, {\it Convergence rate of inertial Forward–Backward algorithm beyond Nesterov's rule}, Math. Program. (2018). https://doi.org/10.1007/s10107-018-1350-9

\bibitem{abbas-att-sv} B. Abbas, H. Attouch, B.F. Svaiter, {\it Newton-like dynamics and forward-backward methods for
structured monotone inclusions in Hilbert spaces}, Journal of Optimization Theory and its Applications 161(2), 331-360, 2014

\bibitem{alv-att-bolte-red} F. Alvarez, H. Attouch, J. Bolte, P. Redont, {\it A second-order gradient-like dissipative dynamical system
with Hessian-driven damping. Application to optimization and mechanics}, Journal de Math\'{e}matiques Pures et Appliqu\'{e}es (9) 81(8),
747-779, 2002

\bibitem{AC} H. Attouch, Z. Chbani,{\it Fast inertial dynamics and fista algorithms in
convex optimization. Perturbation aspects}, arXiv preprint arXiv:1507.01367,
2015
\bibitem{att-c-r-arx2017} H. Attouch, Z. Chbani, H. Riahi, {\it Rate of convergence of the Nesterov accelerated gradient method in the
subcritical case $\alpha\leq 3$}, ESAIM: COCV (2017). https://doi.org/10.1051/cocv/2017083
\bibitem{att-c-r} H. Attouch, Z. Chbani, H. Riahi, {\it Fast convex optimization via time  scaling of damped inertial gradient dynamics}, https://hal.archives-ouvertes.fr/hal-02138954
\bibitem{att-c-p-r-math-pr2018} H. Attouch, Z. Chbani, J. Peypouquet, P. Redont, {\it Fast convergence of inertial dynamics and algorithms with asymptotic
vanishing viscosity}, Math. Program. 168(1-2) Ser. B, 123-175, 2018

\bibitem{AGR} H. Attouch, X. Goudou,  P. Redont, {\it  The heavy ball with
friction method, I. The continuous dynamical system: global exploration of the
local minima of a real-valued function by asymptotic analysis of a dissipative
dynamical system}, Communications in Contemporary Mathematics 2, 1-34, 2000

\bibitem{att-p-r-jde2016} H. Attouch, J. Peypouquet, P. Redont, {\it Fast convex optimization via inertial dynamics with Hessian driven damping},
J. Differential Equations, 261(10), 5734-5783, 2016

\bibitem{att-sv2011} H. Attouch, B.F. Svaiter, {\it A continuous dynamical Newton-like approach to solving monotone inclusions},
SIAM Journal on Control and Optimization 49(2), 574-598, 2011
\bibitem{BM} M. Balti, R. May,{\it Asymptotic for the perturbed heavy ball system
with vanishing damping term}, arXiv preprint arXiv:1609.00135, 2016.
\bibitem{bete09} A. Beck, M. Teboulle, {\it A Fast Iterative Shrinkage-Thresholding Algorithm for Linear Inverse Problems}, SIAM Journal on Imaging Sciences, 2(1), 183-202, 2009
\bibitem{BBJ} P. B\'egout, J. Bolte, M.A. Jendoubi, {\it On damped second-order gradient systems}, Journal of Differential Equations, (259), 3115-3143, 2015

\bibitem{BCL1} R.I. Bo\c t¸, E.R. Csetnek, S.C. L\'aszl\'o, {\it An inertial forward-backward algorithm for minimizing
the sum of two non-convex functions}, Euro Journal on Computational Optimization, 4(1), 3-25, 2016
\bibitem{BCL} R.I. Bo\c t,  E.R. Csetnek,  S.C. L\' aszl\' o, {\it  Approaching nonsmooth nonconvex minimization through
second-order proximal-gradient dynamical systems}, Journal of Evolution Equations, 18(3), 1291-1318, 2018
\bibitem{ALV} C. Alecsa, S.C. L\'aszl\'o, A. Viorel, {\it A gradient type algorithm with backward inertial steps associated to a nonconvex minimization problem}, https://www.researchgate.net/publication/331832475
\bibitem{BCL-AA} R.I. Bo\c t,  E.R. Csetnek,  S.C. L\' aszl\' o, {\it A second order dynamical approach with variable damping to nonconvex smooth minimization}, Applicable Analysis (2018). https://doi.org/10.1080/00036811.2018.1495330

\bibitem{BCL-MM} R.I. Bo\c t,  E.R. Csetnek,  S.C. L\' aszl\' o, {\it A primal-dual dynamical approach to structured convex minimization problems}, https://arxiv.org/abs/1905.08290
\bibitem{CAG1} A. Cabot, H. Engler, S. Gadat, {\it On the long time behavior of second order differential equations with asymptotically small dissi-
pation}, Trans. Amer. Math. Soc. 361, pp. 5983-6017, 2009
\bibitem{CAG2} A. Cabot, H. Engler, S. Gadat,{\it Second order differential equations with asymptotically small dissipation and piecewise
at potentials}, Electronic Journal of Differential Equations 17, 33-38, 2009
\bibitem{ch-do2015} A. Chambolle, Ch. Dossal, {\it On the convergence of the iterates of the "fast iterative shrinkage/thresholding algorithm''},
J. Optim. Theory Appl., 166(3), 968-982, 2015

\bibitem{haraux} A. Haraux, {\it Syst\`{e}mes Dynamiques Dissipatifs et Applications},
Recherches en Math\'{e}matiques Appliqu\'{e}ées 17, Masson, Paris, 1991

\bibitem{L} S.C. L\'aszl\'o, {\it Convergence rates for an inertial algorithm of gradient type associated to a smooth nonconvex minimization}, https://arxiv.org/pdf/1811.09616.pdf

\bibitem{MJ} M. Muehlebach, M. I. Jordan, {\it A Dynamical Systems Perspective on Nesterov Acceleration}, arXiv:1905.07436v1

\bibitem{Nest} Y. Nesterov , {\it Introductory lectures on convex optimization: a basic course}. Kluwer Academic Publishers, Dordrecht, 2004

\bibitem{nesterov83} Y.E. Nesterov, {\it A method for solving the convex programming problem with convergence rate $O(1/k^2)$},
(Russian) Dokl. Akad. Nauk SSSR 269(3), 543-547, 1983

\bibitem{Opial} Z. Opial,{\it  Weak convergence of the sequence of successive approximations for nonexpansive mappings}, Bull. Amer. Math. Soc., 73,
591-597, 1967
\bibitem{poly} B. T. Polyak, {\it Some methods of speeding up the convergence of iteration methods}, U.S.S.R. Comput. Math. Math. Phys., 4(5), 1-17, 1964

\bibitem{sontag} E.D. Sontag, {\it Mathematical control theory. Deterministic finite-dimensional systems}, Second edition,
Texts in Applied Mathematics 6, Springer-Verlag, New York, 1998

\bibitem{su-boyd-candes} W. Su, S. Boyd, E.J. Candes, {\it A differential equation for modeling Nesterov's
accelerated gradient method: theory and insights}, Journal of Machine Learning Research, 17, 1-43, 2016
\bibitem{Bruck} R.E. Bruck, {\it Asymptotic convergence of nonlinear contraction semigroups in Hilbert spaces}, J. Funct. Anal., 18 ,  15-26, 1975
\end{thebibliography}
\end{document}